\documentclass[12pt]{amsart}
\usepackage{microtype}
\usepackage{graphicx}
\usepackage{hyperref}
\usepackage[dvipsnames]{xcolor}
\usepackage{tikz}
\usepackage{amsthm,amsmath,amssymb,amsfonts}
\usepackage{thmtools,thm-restate}
\usepackage[dvipsnames]{xcolor}
\definecolor{DanBlue}{RGB}{0,0,225}

\DeclareMathOperator\SL{SL}
\DeclareMathOperator\DPV{DPV}

\DeclareMathOperator\Rept{Re}
\DeclareMathOperator\Rs{Res}

\def\Real{{\mathbb R}}
\def\rot{{\rho}}
\def\trans{{\mcal{T}}}
\def\bbR{{\mathbb R}}
\def\Sphere{{\mathbb S}}
\def\Isom{{\mathrm{Isom}}}
\def\SO{{\mathrm{SO}}}
\def\bbC{{\mathbb C}}
\def\bbZ{{\mathbb Z}}
\def\bbQ{{\mathbb Q}}
\def\bbN{{\mathbb N}}

\def\Expec{{\mathbf E}\,}
\def\eps{{\varepsilon}}
\def\One{{\mathbf{1}}}

\newcommand{\diff}{\mathop{}\!\mathrm{d}}
\newcommand{\Ft}[1]{{\widehat{#1}}}
\newcommand{\Id}{\text{Id}}
\newcommand{\id}{\text{id}}

\newcommand{\Tr}{\operatorname{Tr}}

\def\lsim{{\,\lesssim\,}}
\newcommand{\mbf}[1]{{\mathbf{#1}}}
\newcommand{\mcal}[1]{{\mathcal{#1}}}
\newcommand{\Prob}{\mathbb{P}}

\newcommand{\gsim}{\gtrsim}

\newcommand{\supp}{{\operatorname*{supp}\,}}

\numberwithin{equation}{section}

\declaretheorem[numberwithin=section]{theorem}
\declaretheorem[sibling=theorem]{lemma,proposition,corollary,conjecture,definition}
\newtheorem{claim}[theorem]{Claim}

\usepackage[backend=biber]{biblatex}
\title{Local limit theorems for random isometries of the plane}
\author{Reuben Drogin}
\address{Department of Mathematics, Yale University, New Haven, CT}
\email{reuben.drogin@yale.edu}
\author{Felipe Hern{\'a}ndez}
\address{Department of Mathematics, Penn State University, State College, PA}
\email{felipeh@psu.edu}
\date{\today}

\begin{document}

\begin{abstract}
We consider a random walk
$(Y_N)_{N\geq 0}$ on $\Real^2$ generated by
successively applying independent random isometries, drawn from a
fixed measure $\mu$,
to the point $0$.
When the support of $\mu$ is finite and includes an irrational rotation satisfying a Diophantine condition,
we establish a local central limit theorem (LCLT) for $Y_N$
down to super-polynomially small scales.
When $\mu$ includes rotations satisfying a further algebraic condition,
we prove that a LCLT holds down to the scale
$\exp(-cN^{1/3}/(\log N)^2)$. Due to group-theoretic obstructions,
this is sharp
for symmetric $\mu$, up to the $\log$ factor.
Lastly for a special class of asymmetric $\mu$,
we obtain an LCLT down to the much finer scale $\exp(-cN^{1/2})$.
The proofs relate the fine-scale distribution of $Y_N$ to a question about the values of
integer polynomials on the unit circle.
\end{abstract}
\maketitle

\section{Introduction }
In this paper we consider the fine-scale distribution of random walks on planar isometries. In particular, we are interested in the distribution of the ``translation part'' of the random isometry $g_1\cdots g_N$ given by
\begin{equation}\label{def:YN}
Y_N := \left(g_1 g_2 \cdots g_N\right)\left(0\right),
\end{equation}
where $g_j\in \Isom(\mathbb{R}^2)$ are independently and uniformly sampled from a fixed
probability measure $\mu$ with finite support.

If $\supp \mu$ contains only pure translation elements,
then the classical Berry-Esseen Central Limit theorem implies that $Y_N$ has the law of a Gaussian distribution
down to unit scales.
That is, for each $N$ there is a Gaussian $Z_N$ such that
for all $r\gsim 1$ and $x_0\in \mathbb{R}^2$
\begin{equation}\label{eq:local-lim}
\begin{aligned}
\left|\Prob\left(Y_N\in B_r(x_0)\right) - \Prob\left(Z_N\in B_r(x_0)\right)\right| = o\left(N^{-1} r^2\right),
\end{aligned}
\end{equation}
More generally, we say that $Y_N$ satisfies a local central limit theorem (LCLT)
to scale $\delta_N$ if for each $N$ there exists a Gaussian $Z_N$,
such that~\eqref{eq:local-lim} holds uniformly for $r>\delta_N$ and $x_0\in \mathbb{R}^2$.
In this paper we investigate how the presence of rotations greatly improves
the scale at which ~\eqref{eq:local-lim} holds.

Group-theoretic considerations place significant constraints on the lower bound for $\delta_N$.
The first and most obvious restriction is that $\delta_N$ can decay only polynomially when the
group generated by $\supp \mu$ has only polynomial growth.  This happens in particular when
all rotation angles in $\supp\mu$ are rational multiples of $\pi$.
See Figure~\ref{fig:rational-vs-irrational} for an illustration of
the difference in this case versus when $\supp\mu$ contains an irrational rotation.

\begin{figure}
\begin{tabular}{cc}
\includegraphics[scale=0.1]{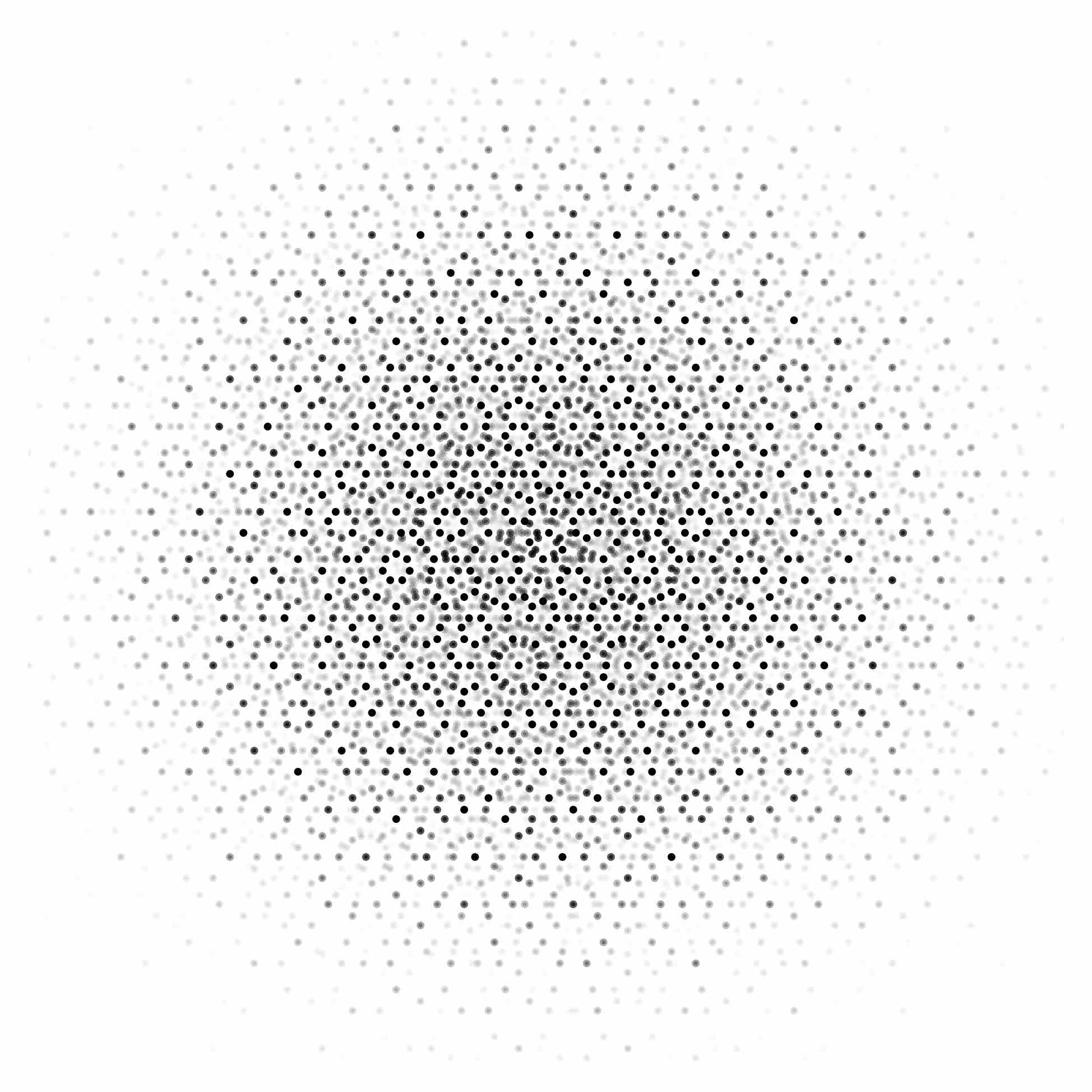} &
\includegraphics[scale=0.1]{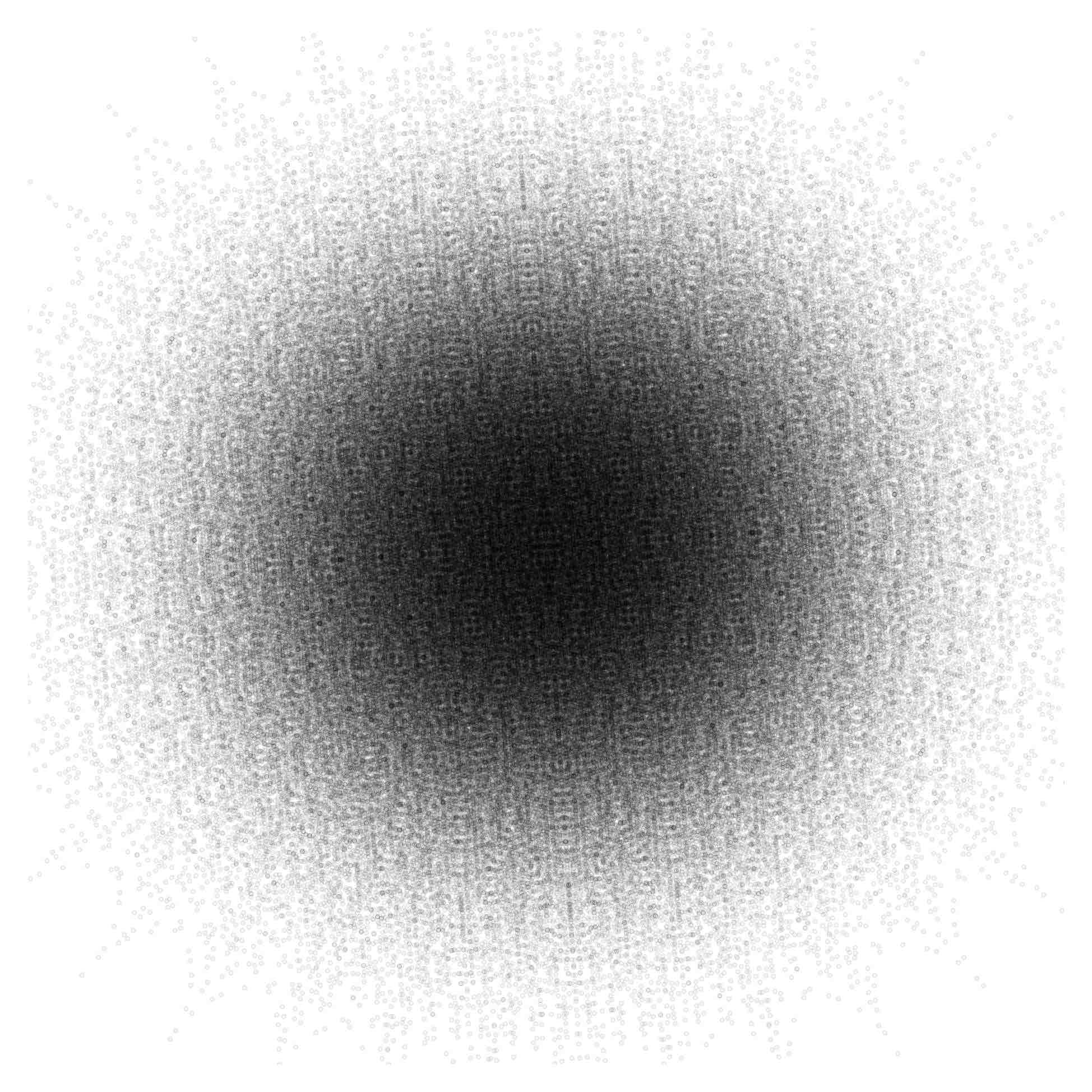}
\end{tabular}
\caption{Plot of
$\sum_{j=0}^{17} \eta_j z^j$ with $\eta_j\in\{\pm1\}$ uniformly and independently,
with $z=e^{2\pi i/7}$ on the left and $z=e^{\sqrt{2} \pi i}$ on the right.  This corresponds to a random isometry with
$\mu$ uniform on $\{\tau_{\pm 1}\circ\rho_\theta\}$.  The distribution of $Y_N$ resembles a Gaussian down to much finer scales when
irrational rotations are included.
}
\label{fig:rational-vs-irrational}
\end{figure}

Our first result shows that, indeed, $\delta_N$ decays super-polynomially
when $\supp\mu$ has elements with
sufficiently irrational rotation part.  To state this result, we say an angle
$\theta\in [0,2\pi]$ satisfies the Diophantine condition \eqref{eq:DCm} if
\begin{equation}\label{eq:DCm}
\tag{D.C.-$m$}
\textit{ for all $p\in \mathbb{Z}$ and large $q\in \mathbb{N}$, }\qquad|\theta - \frac{2\pi p}{q}|\geq q^{-m}.
\end{equation}
We also use $\theta(g)$ to denote the angle of rotation of $g\in \Isom(\mathbb{R}^2)$.

\begin{theorem}\label{thm:main-diophantine}
Let $\mu$ be a measure on $\Isom(\mathbb{R}^2)$ with finite support,
and $Y_N$ be given as in \eqref{def:YN}.
Suppose also that $\mathbb{E}g_{1}(0) = 0$ and one of the following holds with some $m>0$:
\begin{enumerate}
\item There exists $g_1,g_2\in \supp\mu$ such that $\theta(g_1g_2^{-1})$ satisfies \eqref{eq:DCm}.
\item For some $\theta_0$ satisfying \eqref{eq:DCm}, $\theta(g)=\theta_0$ for all $g\in\supp \mu$.
\end{enumerate}
Then for some $c>0$,
$Y_N$ satisfies an LCLT to scale $e^{-c(\log N)^2}$. More precisely, with $\sigma^2 := \frac12 \Expec |Y_0|^2$
we have for any $x_0\in\Real^2$ and $r>\exp(-c(\log N)^2)$ the estimate
\begin{align}
\label{eq:LCLT-main}
|\Prob(Y_N\in B_r(x_0))  - \Prob(\sigma \sqrt{N} Z_1 \in B_r(x_0))| \leq C N^{-3/2} r^2,
\end{align}
where $Z_1$ is the standard Gaussian on $\Real^2$.
\end{theorem}

\begin{figure}[t]
\begin{tabular}{cc}
\includegraphics[scale=0.1]{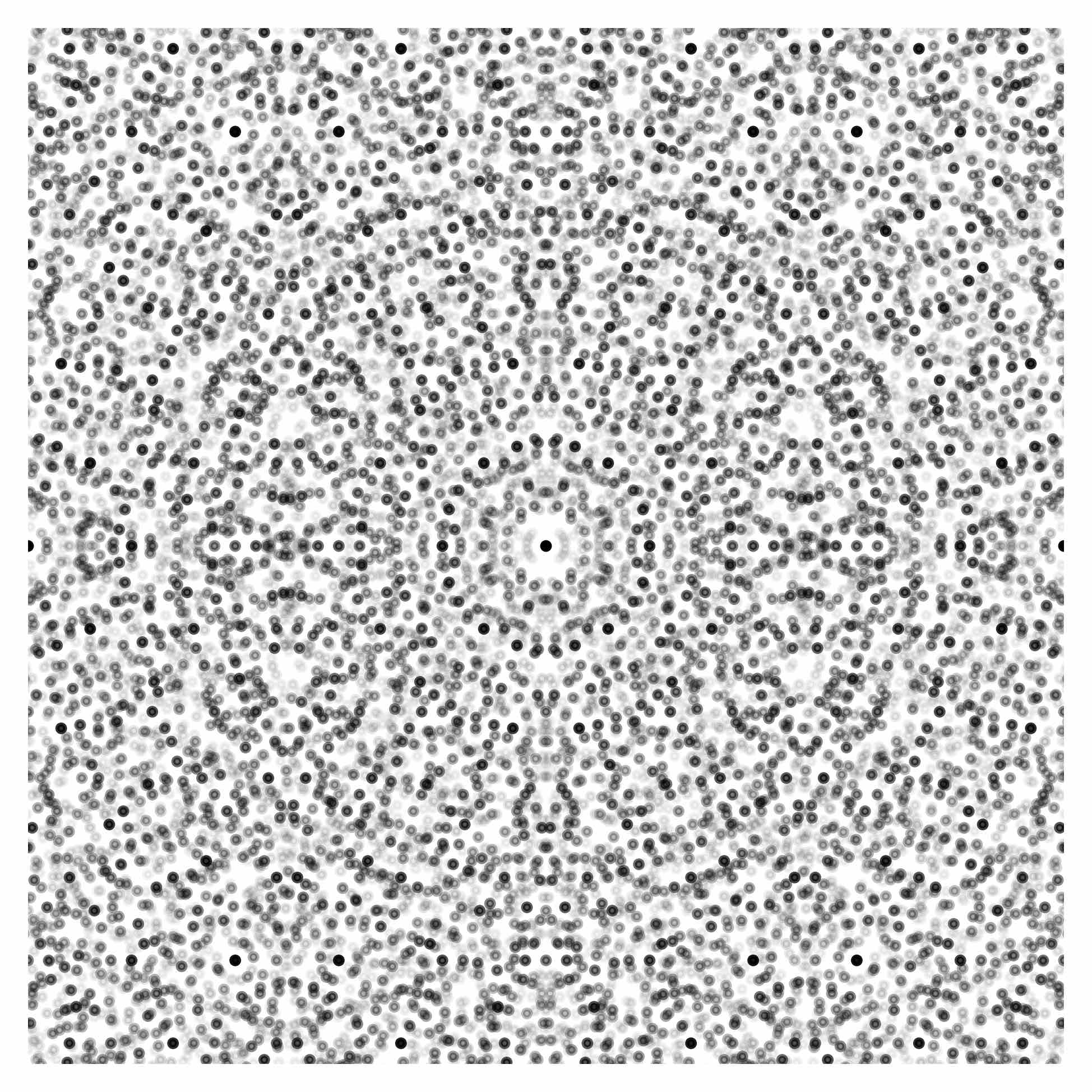} &
\includegraphics[scale=0.1]{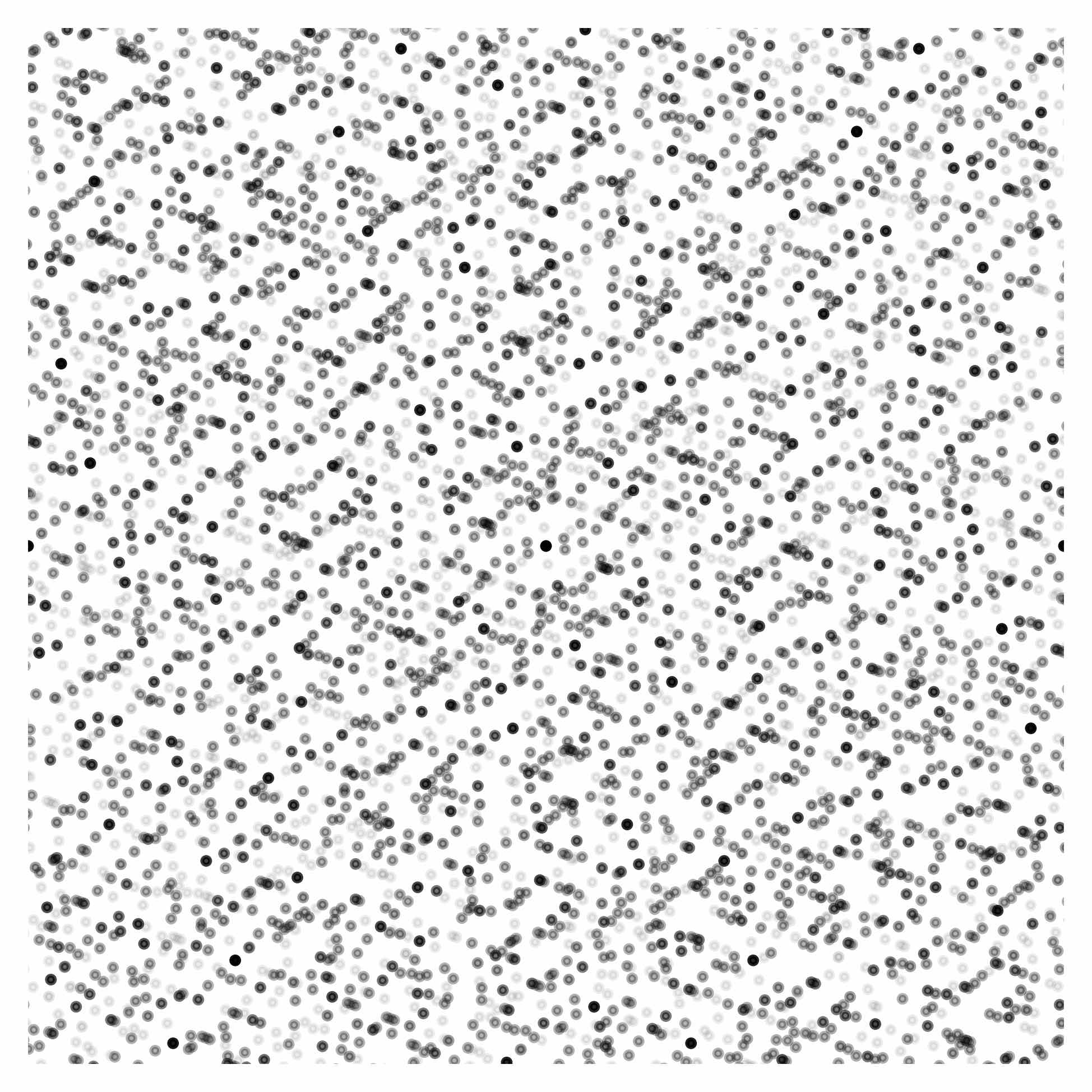}
\end{tabular}
\caption{The point distributions for $Y_N\cap [-1,1]^2$ when $\mu$ is uniformly supported on the symmetric set
$\{\rho_\theta,\rho_{-\theta},\tau_1,\tau_{-1}\}$ (left) and the asymmetric set $\{\rho_\theta,\tau_1,\tau_{-1}\}$
(right).  Darker dots indicate points with higher multiplicity.
The image above was generated with $N=13$ and $\theta$ such that $e^{i\theta} = \frac15(3+4i)$.  For the symmetric case, $\Prob(Y^{\rm sym}_N=0)\approx 0.02$, whereas in the asymmetric case we have $\Prob(Y^{\rm asym}_N=0)\approx 0.0002$.  The asymmetric
$Y^{\rm asym}$ has a noticeably more uniform distribution.}
\end{figure}

In the second case of Theorem~\ref{thm:main-diophantine} above, $Y_N$ has the law
of the random polynomial sum $\sum_{j=0}^{N-1} \sigma_j z^j$, where $z=e^{i\theta_0}$
and $\sigma_j$ are independently sampled copies of $Y_1\overset{d}{=}g_1(0)$.

One might expect that when $\supp\mu$ generates a group of exponential growth,
then $\delta_N$ might be taken to be exponentially small with $N$.  However, there
is a deeper group-theoretic obstruction which rules this out when $\mu$ is symmetric,
i.e. $\Prob(g\in S) = \Prob(g^{-1}\in S)$ for all $S\subseteq\Isom(\Real^2)$.

We can illustrate this obstruction best in the case that $\mu$ is uniform on $\{\rho_{\pm \theta},\tau_{\pm 1}\}$.  In this case, the group generated by
$\supp \mu$ acts on $(f,\ell)\in\bbZ^\bbZ\times \bbZ$ as follows:
\begin{align*}
\rho_{\pm\theta} (f,\ell)  &= (f,\ell\pm 1) \\
\tau_{\pm 1} (f,\ell)  &= (f\pm e_{\ell}, \ell), \text{ where $e_{\ell}(k) = \delta_{k=\ell}$.}
\end{align*}
Then setting $(f_N,\ell_N) := (g_N\cdots g_1)(\mbf{0},0)$ we have the identity
\[
Y_N = (g_1...g_N)(0) = \sum_{j=-\infty}^\infty f_N(j) e^{ij\theta}.
\]
The random variable $\ell_n$ performs a random walk on $\bbZ$,
and thus we have $\supp f_{N}$ is contained in $[-k,k]$ with probability on the order $\exp(-cN/k^2)$.
For each $j\in[-k,k]$, $f_N(j)$ is the endpoint of a random walk on $\bbZ$ having at most $N$ steps,
so is identically zero with probability\footnote{Ignoring parity concerns.} at least $N^{-1/2}$.  Taking $k=N^{1/3}$,
we have that $f_N\equiv 0$ with probability at least
$N^{-cN^{1/3}} = \exp(-cN^{1/3}\log N)$.  Hence, for this specific $\mu$,
one cannot have a LCLT to scale $\delta_N < \exp(-cN^{1/3}\log N)$.

In a more group-theoretic language, the obstruction in the example above comes from
the fact that there is a surjective homomorphism from the the wreath product $\bbZ\wr\bbZ$
to the group with generators $\{\rho_{\pm\theta},\tau_{\pm v}\}$.
A result of Saloff-Coste and Pittet~\cite[Theorem 3.11]{pittet2002random}
shows that the random walk $h_N:= g_1\cdots g_N$ on the Cayley graph of $\bbZ\wr\bbZ$ satisfies
\[
\Prob(h_N = \id) \geq \exp(-cN^{1/3}(\log N)^{2/3}).
\]
In general, the group generated by any $\mu$ on $\Isom(\mathbb{R}^2)$ is amenable (since the
commutator subgroup of $\Isom(\mathbb{R}^2)$ is the group of translations which is abelian).
Thus, when $\mu$ is symmetric, a classical result of Kesten~\cite{kestenbanach,kestensymmetric}
shows $\lim_{N\to\infty} \Prob(h_{2N}=\id)^{1/2N} =1$. Hence, for \textit{any}
symmetric $\mu$, one cannot have an LCLT down to exponentially small scales.

Our next result establishes a local limit theorem down to the scale $\exp(-cN^{1/3}/\log N)$ for any symmetric $\mu$ such that $\supp\mu$ has an element which rotates by an irrational angle with a rational cosine.
\begin{theorem}
\label{thm:main-symmetric}
Let $\mu$ be a symmetric probability measure with finite support, and $Y_N$ be as in \eqref{def:YN}.
If there exists $g\in \supp\mu$ with $\cos\theta(g)\in\bbQ \setminus\frac12\bbZ$,
then $Y_N$ satisfies an LCLT down to scale
$\exp(-cN^{1/3}/(\log N)^2)$.  That is, the estimate~\eqref{eq:LCLT-main} holds for all $r>\exp(-cN^{1/3}/(\log N)^2)$.
\end{theorem}

In the absence of symmetry, we are not aware of any reason that $\delta_N$ should not decrease as
$\exp(-cN)$.  We are unable to prove any result down
to an exponentially fine scale, but we are able to beat the $\exp(-N^{1/3})$ scale in two special examples
\begin{enumerate}
\item \textbf{Littlewood polynomials}: The polynomial case is that $\mu$ is uniformly distributed on $\{\tau_1\rho_\theta, \tau_{-1}\rho_\theta\}$.  In this case, $Y_N$ has the distribution of the random polynomial $\sum_{j=0}^{N-1} \sigma_j z^j$ with $z=e^{i\theta}$
and $\sigma_j$ independent and uniform samples from $\{\pm1\}$.
\item \textbf{Asymmetric wreath product}: We say that $\mu$ generates an asymmetric wreath product if
$\supp\mu = \{\rho_\theta,\rho_{-\theta},\tau_1,\tau_{-1}\}$ with $\Prob(g=\tau_1)=\Prob(g=\tau_{-1})$, but $\Prob(g=\rho_\theta) \not= \Prob(g=\rho_{-\theta})$.
\end{enumerate}
\begin{theorem}
\label{thm:main-asymmetric}
Let $\mu$ be as above (either a Littlewood polynomial or an asymmetric wreath product),
with $\cos\theta \in \bbQ\setminus\frac12\bbZ$ and $Y_N$ be as in \eqref{def:YN} for this $\mu$.
Then the LCLT~\eqref{eq:LCLT-main} holds for any $r > \exp(-c\sqrt{N}/\log N)$.
\end{theorem}

In the next subsection we give an overview of the main ideas of the proof, and conclude the introduction with a discussion of related
works in Section~\ref{sec:related}.

\subsection{Commentary on the proof}
\label{sec:overview}
To prove a local limit theorem to scale $\delta$ it suffices
(see Lemma~\ref{lem:llerr}) to prove the following estimate on the characteristic function of $Y_N$:
\begin{equation}
\label{eq:FT-bd}
\int_{|\xi| < \delta^{-10}}
|\Expec e^{-i\langle \xi,Y_N\rangle}
- e^{-\frac12 \sigma^2 N|\xi|^2}|
\diff \xi = o(N^{-1}),
\end{equation}
where $\sigma^2 := \frac12 \Expec |Y_0|^2$.
The low frequency part, say $|\xi|\ll 1$, can be estimated using standard arguments in probability theory, and
does not rely on any diophantine assumption on any rotation angles.

In contrast, the high frequency regime is much more delicate. An illustrative example is the
case that $g$ is uniformly sampled from the isometries $x\mapsto e^{i\theta}x\pm 1$.  Setting $z=e^{i\theta}$, $Y_N$ has the same
law as the random polynomial $p(z) = \sum_{j=0}^{N-1} \sigma_j z^j$ with $\sigma_j\in\{\pm1\}$ independently.
A simple Fourier-analytic calculation shows that the characteristic function
evaluated at $\xi=r e^{i \varphi}$ has the form
\begin{align}
\label{eq:LP-FT}
\Expec e^{-i\langle \xi,Y_N\rangle} = \prod_{j=1}^N \cos(r\cos(j\theta+\varphi)).
\end{align}
The values $\{\cos(j\theta+\varphi)\}_{j=1}^N$ are the real parts of $\{e^{i(j\theta+\varphi)}\}_{j=1}^N$, which are spaced by about $N^{-1}$ on the unit circle.  The spacing between almost all adjacent values of $\cos(j\theta+\varphi)$ is therefore typically around $N^{-1}$, and the minimal spacing is
about $N^{-2}$ near $\pm 1$.
Using only this metric information, it is impossible to rule out the possibility that
$\{\cos(j\theta+\varphi)\}_{j=1}^N$ is contained in an arithmetic progression of spacing $N^{-2}$.
If this were the case, the product above could be identically $1$ for frequencies of order $N^2$.  Therefore, to obtain decay of the characteristic function for frequencies much larger than $N^2$, it is necessary to use additional information.

Our main idea is that we can directly rule out additive structure of the set
$\{\Rept(e^{i\varphi}z^j)\}_{j=1}^N$ in terms of a question about the values of $p(z)$
for general polynomials of low-weight.  This reformulation allows us to use simple algebraic arguments, and is explained below in Section~\ref{sec:DPV-xp}.
The second idea, which applies in the special setting of Theorem~\ref{thm:main-asymmetric}, is based on an arithmetic argument involving a recurrence for $\cos(j\theta+\varphi)$ and is discussed in Section~\ref{sec:arithmetic}.

\subsubsection{Polynomial values of $e^{i\theta}$}
\label{sec:DPV-xp}
The first idea, which we employ both in the proof of Theorem~\ref{thm:main-diophantine} and in
Theorem~\ref{thm:main-symmetric}, is that $\cos(j\theta+\varphi)=\Rept(e^{i\varphi}z^j)$
cannot lie very close to some lattice $\delta\bbZ$ if
many simple linear combinations of $\cos(j\theta+\varphi)$ produce values in the range $[\delta/10, \delta/2]$.
This is captured in the following definition.
\begin{definition}[Dense polynomial values]
\label{def:DPV}
The angle $\theta\in[0,2\pi)$ satisfies the \emph{dense polynomial value property}
$\DPV(n,R)$ if the set
\[
\Bigg\{\log \Big(\sum_{j=0}^D c_j e^{ij\theta}\Big) \mid  c_j\in\bbZ, \,\,\sum_{j=0}^D (1+|c_j|) \leq n\Bigg\}
\]
forms a $0.1$-net in $[-\log R, 0]\times[0,2\pi]\subset \bbC$.
We also define
\[
R_{\DPV}(n,\theta) := \sup \{R\in\bbR \mid \theta \textrm{ satisfies }
\DPV(n,R) \}.
\]
\end{definition}
Note that a simple \textit{necessary} condition for the LCLT to hold to scale $\delta_N$ is that $R_{\DPV}(N,\theta) \gsim \delta_N^{-1}$.
What we show is that this is not too far from also being a \textit{sufficient}
condition.  Indeed, Theorem~\ref{thm:main-diophantine} follows from the following super-polynomial bound on $R_{\DPV}$ for irrational angles.
\begin{restatable}{proposition}{superpoly}
\label{prp:super-poly}
If $\theta$ satisfies the condition ~\eqref{eq:DCm},
then there exists $c>0$ such that $\theta$ satisfies $\DPV(n,R)$ for all sufficiently large $n$ and $R < \exp(c(\log n)^2)$.  That is,
for all sufficiently large $n$
\begin{align}
\label{eq:R-lbd}
R_{\DPV}(n,\theta) \geq \exp(c (\log n)^2).
\end{align}
\end{restatable}
The idea of the proof of Proposition~\ref{prp:super-poly} is to use explicit test polynomials of the form
$(z^q-1)^k$ for appropriate $q$ such that $|z^q-1| < n^{-c}$.  Analytically, this corresponds to taking high order
finite differences of closely and uniformly spaced points on a circle.

The lower bound~\eqref{eq:R-lbd} is nevertheless far from the exponential upper bound which
one may expect to hold generically.  We are however able to achieve such a lower bound in the special case $\cos\theta\in\bbQ\setminus\frac12\bbZ$.  In this case, $z=e^{i\theta}$ is algebraic with minimal polynomial of the form
\begin{align}
\label{eq:minimal-poly-intro}
p_0(z) = az^2 - bz + a
\end{align}
with $\gcd(a,b)=1$ and $a>1$.
This algebraic relation is the key to the following result.
\begin{restatable}{proposition}{algebraic}
\label{prp:exp-Rbds}
Suppose $\theta\in[0,2\pi)$ with $\cos\theta\in\bbQ\setminus\frac12\bbZ$.  Then
there exists $c>0$ such that $\theta$ satisfies $\DPV(n,R)$ for all $n$ large
and $R<\exp(cn/(\log n)^2)$.  That is, for all sufficiently large $n$
\begin{align}
\label{eq:DPV-exp}
R_{\DPV}(n,\theta) &\geq \exp(cn/(\log n)^2)
\end{align}
\end{restatable}
The bound~\eqref{eq:DPV-exp} is used to prove Theorem~\ref{thm:main-symmetric}.  The
argument which bounds the characteristic function in terms of $R_{\DPV}$ goes through a
spectral-theoretic argument used by Varj\'u in his work on three-dimensional isometries~\cite{varju2015random}.  This argument is responsible for the appearance of the
$N^{1/3}$ in the exponential in Theorem~\ref{thm:main-symmetric}

We conjecture that a bound of the form~\eqref{eq:DPV-exp} is generic for $\theta\in[0,2\pi)$.
\begin{conjecture}
\label{conj:theta-generic}
For generic $\theta\in[0,2\pi)$, there exist $c>0$ such that
for all sufficiently large $n$
\begin{align*}
R_{\DPV}(n,\theta) &\geq \exp(cn)
\end{align*}
\end{conjecture}
Conjecture~\ref{conj:theta-generic} appears to be related to questions of algebraic approximations to transcendental numbers in
the spirit of Sprindz\v uk's conjecture \cite{MR548467,MR1652916}.  The difference is that Sprindz\v uk's conjecture concerns approximations
with increasing height and bounded degree, and we are interested in polynomials with bounded height and increasing degree.

\subsubsection{Arithmetic structure of $\cos(j\theta+\varphi)$}
\label{sec:arithmetic}
The second idea we use to handle the product~\eqref{eq:LP-FT} is based on arithmetic properties of a recurrence relation.
For any $r>0$ and $\theta,\varphi\in[0,2\pi)$, we have the following recurrence
for $y_j := r\cos(j\theta+\varphi)$:
\[
y_{j-1} + y_{j+1} = 2\cos\theta y_j.
\]
One would like to show that $y_j$ cannot be too close to an integer for many $j$.  A natural idea is to consider
the dynamical system
\[
\begin{pmatrix}
y_j \\ y_{j+1}
\end{pmatrix}
=
\begin{pmatrix}
0 & 1 \\
-1 & 2\cos\theta
\end{pmatrix}
\begin{pmatrix}
y_{j-1} \\ y_{j}
\end{pmatrix}.
\]
If the transfer matrix above were a hyperbolic element of $\SL(2,\Real)$, it would be easy to show that
$y_j$ could not be an integer too often.  Of course, since $2\cos\theta < 2$, this is never the case.  However,
the transfer matrix \textit{does} behave like a hyperbolic element in a $p$-adic sense when $2\cos\theta\in\bbQ\setminus\bbZ$.
Indeed, if $y_0\in\bbQ$, then the $p$-adic norm of $y_j$ is exponentially decreasing as $|j|\to\infty$ for primes $p$ dividing the
denominator of $2\cos\theta$.

\subsection{Related works}
\label{sec:related}
Random walks on $\Isom(\Real^2)$ were considered in the early work of Ka\v zdan~\cite{KazdanUniform},
who established an ``equidistribution'' result for $Y_N$.
This was made more quantitative and generalized by Guivarch~\cite{guivarch}, who established a central limit theorem
at macroscopic scales for $Y_N$.  Khoklov~\cite{KhokhlovLocalLimit},
then established a local limit theorem for $Y_N$ with
under the assumption that $\mu$ is absolutely continuous.
Other than these works, we are unaware of results
that establish local limit theorems for $Y_N$ directly, at least in dimension $2$.

When $d\geq 3$, much finer results are known
and we briefly mention them here. In his thesis~\cite{varju2015random},
which directly motivated this paper, Varj\'u showed that one can take
$\delta_N = \exp(-cN^{1/4})$ for generic choices
of $\mu$.  Varj\'u's argument crucially used that the rotation group $\SO(d)$
is non-commutative when $d\geq 3$.  In particular,~\cite{varju2015random}
makes crucial use of the Solovay-Kitaev algorithm~\cite{kitaev1997quantum, dawson2006solovay}
for $\SO(d)$, which is equivalent to a so called restricted spectral gap.
This was improved in a later work of Varj\'u and Lindenstrauss~\cite{lindenstrauss2016random} to the optimal exponential scale $\delta_N=\exp(-cN)$ assuming that the rotational parts of the elements of $S$ satisfy
a true \textit{spectral gap} condition.
Bourgain and Gamburd~\cite{bourgain2010spectral} verified this spectral gap condition for rotation matrices satisfying an algebraic
condition.
Establishing a spectral gap for generic rotations remains a major open problem.
By comparison, the group $\SO(2)$ is commutative, and so no set of generators
can satisfy the analogue of the spectral gap condition of~\cite{lindenstrauss2016random}.
This is why we need to use the interaction between the rotation and translation parts of group elements, which is encoded in the
``dense polynomial values'' property.

In the special case that $\supp(\mu) = \{\tau_{+1} \circ \rho_\theta, \tau_{+1} \circ \rho_\theta\}$, $Y_N$
has the same law as the random polynomial
\[
Y_N \stackrel{d}{=} p(e^{i\theta}) = \sum_{j=0}^N \sigma_j e^{ij\theta}
\]
where the $\sigma_j$ are $+1/-1$ Bernoulli random variables.  Motivated by analogy to random matrix theory, one interesting question about such polynomials is the
distribution of the roots of $p$.  Universality for the root distribution was proven by Tao and Vu in~\cite{TaoVu},
where a key ingredient was an anticoncentration bound for $\log |p(z)|$.
In~\cite{TaoVu}, such an anticoncentration result was obtained by combining
deep results of Nguyen and Vu~\cite{nguyen2011optimal}, from inverse Littlewood-Offord theory,
and Shalom and Tao ~\cite{shalom-tao} which quantitatively refines Gromov's theorem on groups of polynomial growth.  Our first result, Theorem~\ref{thm:main-diophantine}
is closely related to Lemma 14.1 from~\cite{TaoVu} in that it also proves an anticoncentration result for the values of Kac polynomials away from roots of unity.
Our proof however avoids inverse Littlewood-Offord theory or quantiative versions of Gromov's theorem and provides a sharper bound for the minimal scale
($\exp(-c(\log n)^2)$ rather than $n^{-C}$).  It it unclear to us however whether Theorem~\ref{thm:main-diophantine}
has any bearing on the universality results for the distributions of zeros.

Another deep question about random polynomials is whether they are
irreducible with high probability~\cite{OdlyzkoPoonen}.  There has been much recent progress on this question~\cite{Kozmairred,BVirred,polyirr}.  As shown by Breuillard and Varj\'u in~\cite{BVirred}, the irreducibility of the random polynomial $p$ is closely related to
a finite field analogue of the random isometries we consider~\cite{BVent,BVcut,MR4637451}.  This in turn is closely related to work
by Breuillard and Varj\'u~\cite{BVconv} on Bernoulli convolutions (see~\cite{varju2018recent} for a nice survey).  The relationship is that
Bernoulli convolutions concern \textit{infinite} sums of \textit{real} numbers of the form $\sum_{j=0}^\infty \sigma_j \lambda^j$,
whereas random polynomials are \textit{finite} sums of \textit{complex} numbers (although see~\cite{shmerkin2016absolute,shmerkin2016absolute2} for work on
complex Bernoulli convolutions).

\subsection{Notation}
We write $\tau_v \in\Isom(\Real^2)$ for translation by $v\in\Real^2$
and $\rot_\theta$ for rotation by $\theta$.  For an isometry $g\in\Isom(\Real^2)$ we write
$\theta(g)$ for the rotation part, and $v(g) = g(0)$ for the translational part, so that
\[
g = \tau_{v(g)} \circ \rot_{\theta(g)}.
\]
We use $\langle v,w\rangle$ for the standard dot product of vectors on $\bbR^2$.
When convenient, we also identify $\bbR^2$ with $\bbC$.
The inner product on complex numbers then takes the form
$\langle x,y\rangle = \Rept (xy^*)$, and we can write the action of $g$ on $\bbC$ as
\[
g(x) = e^{i\theta(g)}x + g(0).
\]

For a polynomial $p(j) = \sum_{j=0}^D c_j z^j$, we write $\|p\|_{\ell^q}$ for the $\ell^q$ norm of its coefficients,
\[
\|p\|_{\ell^q} := \Big(\sum_{j=0}^D |c_j|^q\Big)^{1/q}.
\]

We use $c, c'$ for a small constant that may change from line to line and similarly $C$ is a large constant which may
change from line to line.

\subsection{Structure of the paper}
The rest of the paper is organized as follows.  In Section~\ref{sec:low-freq} we prove an estimate valid at low
frequencies $|\xi|\lsim 1$ which establishes the local limit theorem down to unit scale.  Then in Section~\ref{sec:spectral} we explain the connection between the $\DPV$ property and the characteristic function of $Y_N$ at high frequency.  In Section~\ref{sec:poly-reduction} we prove Proposition~\ref{prp:super-poly} and deduce Theorem~\ref{thm:main-diophantine} as a consequence.  In Section~\ref{sec:symmetric} we prove Proposition~\ref{prp:exp-Rbds} and deduce Theorem~\ref{thm:main-symmetric}.  Finally, in Section~\ref{sec:aadic} we complete the proof of Theorem~\ref{thm:main-asymmetric}.

\subsection*{Acknowledgements}
We would like to thank Sebastian Hurtado Salazar for telling us about the problem,
and for pointing out to us the group-theoretic obstruction
to exponential-scale limit theorems.
FH is supported by NSF award DMS-2303094.

\section{Low Frequency Estimate}
\label{sec:low-freq}
In this section we prove the following result for the characteristic
function of $Y_N$. It is useful at low frequencies, i.e. $|\xi| \ll 1$.
\begin{proposition}
\label{prp:lowfreq}
Let $\mu$ be a measure on $\Isom(\mathbb{R}^2)$ with finite support and
$Y_N$ be as in $\eqref{def:YN}$. If $\mathbb{E}g_1(0) =0$,
then
$$\mathbb{E}e^{-i\langle\xi, Y_N\rangle} = e^{-\frac{1}{2}\sigma^2 N|\xi|^2 + O(N |\xi|^4 + N|\xi|^3)},$$
where $\sigma^2 := \frac{1}{2}\mathbb{E}|g_1(0)|^2.$
\end{proposition}

The proof of this exploits the martingale structure of $Y_N$.
Indeed, if for $j=1$ we define $\rho_{<1} := \Id$, and for $j\in [2,N]$ we define
\[
\rho_{<j}:= \rho_{\theta(g_1)}\rho_{\theta(g_2)}\cdots\rho_{\theta(g_{j-1})}.
\]
Then
\begin{equation}\label{eq:sum-expression}
    Y_N =  \sum_{j=1}^N \rho_{<j}v(g_j) = \sum_{j=1}^N w_j,
\end{equation}
where
$$w_j:= \rho_{<j}v(g_j),$$
for $j\in [1,N]$. Note this follows
since each $g_i$ acts by $g_i(x) = \rho_{\theta(g_i)}x + v(g_i).$

This has a martingale structure since $\mathbb{E}g_1(0) = \mathbb{E}v(g_1) = 0$, so that
\begin{equation}\label{eq:martingale-property}
\mathbb{E}\left[w_{j}\mid g_1,...,g_{j-1}\right] = \rho_{<j}\mathbb{E}\left[ v(g_j)\mid g_1,...,g_{j-1}\right]  = 0.
\end{equation}
The following is a standard calculation for the characteristic function of
martingales.
\begin{lemma}\label{lem:martingale-char-func-comp}
For any $j\in [1,N]$, define $\mathcal{F}_j := \sigma(g_1,...,g_j).$
For any $\xi\in \mathbb{R}^2$ we have
$$\mathbb{E}\exp\left(i\langle\xi, Y_N\rangle\right) = \mathbb{E}\exp\left(-\frac{1}{2}\sum_{j=1}^{N}\mathbb{E}\left[\langle \xi, w_j\rangle^2\ | \mathcal{F}_{j-1}\right] + O(N|\xi|^3)\right).$$
\end{lemma}
\begin{proof}
Since $Y_j$ is $\mathcal{F}_j$-measurable,
iterating equation \eqref{eq:martingale-property} implies
\begin{align*}
\mathbb{E}e^{i\langle\xi, Y_N\rangle}
& = \mathbb{E}\left[e^{i\langle \xi, Y_{N-1}\rangle} \mathbb{E}\left[e^{-i\langle\xi, w_{N}\rangle}\ | \mathcal{F}_{N-1}\right]\right]\\
& = \mathbb{E}\left[e^{i\langle \xi, Y_{N-1}\rangle} e^{-\frac{1}{2}\mathbb{E}\left[\langle \xi, w_{N}\rangle^2\ | \mathcal{F}_{N-1}\right] + O(|\xi|^3)}\right]\\
& = \mathbb{E}e^{-\frac{1}{2}\sum_{j=1}^{N}\mathbb{E}\left[\langle \xi, w_j\rangle^2\ | \mathcal{F}_{j-1}\right] + O(N|\xi|^3)}.
\end{align*}
\end{proof}

We estimate the first term by reducing it to the
Birkhoff sum of a smooth function over the Markov chain
generated by the rotation parts of the $g_i$.
Indeed, define the function $F:\mathbb{R}^2\to \mathbb{R}^2$ by
\begin{equation}\label{def:F}
    F(\xi) := \mathbb{E}_{g_1} \langle \xi, v(g_1)\rangle^2
\end{equation}
and define
\[
S_N:= \sum_{j=1}^{N}\mathbb{E}\left[\langle \xi, w_j\rangle^2| \mathcal{F}_{j-1}\right].
\]

\begin{lemma}\label{lem:reduction-to-markov-chain}
For any $\lambda>0$ we have
$$\mathbb{P}\left(\left|S_N-\sum_{j=1}^N F(\rho_{<j}^{-1}\xi)\right|\geq \lambda \right)\leq 4e^{-c\frac{\lambda^2}{N|\xi|^4}}.$$
\end{lemma}

\begin{proof}
Averaging the $j$-th term over $g_j$ gives $F(\rho_{<j}^{-1}\xi)$.
Hence subtracting $F(\rho_{<j}^{-1}\xi)$ from each term gives a
martingale with bounded increments, and the estimate follows by Azuma's inequality. Indeed, since $w_j = \rho_{<j}v(g_j)$ we have
$$\mathbb{E}_{g_j}\left[\mathbb{E}\left[\langle \xi, w_j\rangle^2| \mathcal{F}_{j-1}\right]\right]
= \mathbb{E}_{g_j}\left[\mathbb{E}\left[\langle \rho_{<j}^{-1}\xi, v(g_j)\rangle^2| \mathcal{F}_{j-1}\right]\right]
= F(\rho_{<j}^{-1}\xi)$$
for each $j\in [1,N]$. Thus
\begin{align*}
\mathbb{P}\left(\left|S_N-\sum_{j=1}^N F(\rho_{<j}^{-1}\xi)\right|\geq \lambda \right)
& = \mathbb{P}\left(\left|\sum_{j=1}^{N}\mathbb{E}\left[\langle \xi, w_j\rangle^2| \mathcal{F}_{j-1}\right] - \mathbb{E}_{g_j}\mathbb{E}\left[\langle \xi, w_j\rangle^2| \mathcal{F}_{j-1}\right]\right|\geq\lambda\right)
\end{align*}
The right hand side can now be estimated by Azuma's inequality.
It is by definition a martingale and each
difference is bounded by $C|\xi|^2$ since $\mu$ has finite support.
Thus the lemma follows by Azuma's inequality.
\end{proof}

Now we show the Birkhoff type sum $\sum_{j=1}^N F(\rho_{<j}^{-1}\xi)$ is concentrated
around $\sigma^2 N|\xi|^2$.
\begin{lemma}\label{lem:markov-chain-estimate}
For any $\lambda>0$ we have
\[
\mathbb{P}\biggl(\Big|\sum_{j=1}^{N}F(\rho_{<j}^{-1}\xi) - \sigma^2 |\xi|^2N\Big|\geq \lambda \biggr)\leq \exp\left(-c \frac{\lambda^2}{N|\xi|^4}\right)
\]
\end{lemma}
\begin{proof}
By the definition of $F(\xi)$, see \eqref{def:F},
we have
$$F(\xi) = \mathbb{E}_{g_1}\langle\xi, \left(\mathbb{E}v(g_1)v(g_1)^\intercal\right)\xi\rangle = \Tr\left((\xi\xi^{T})( \Expec_{g_1} v(g_1) v(g_1)^\intercal)\right).$$
Since $\Tr(\Expec_{g_1} v(g_1) v(g_1)^\intercal) = \mathbb{E}|g_1(0)|^2 = 2\sigma^2$ it suffices to prove
\begin{equation}\label{eq:reduction-1}
\mathbb{P}\left(\left| \sum_{j=1}^N (\rho_{<j}^{-1}\xi)(\rho_{<j}^{-1} \xi)^\intercal - \frac12 N|\xi|^2\Id\right|\geq \lambda\right)\leq Ce^{-c\frac{\lambda^2}{N|\xi|^4}}.
\end{equation}
One can take the $\ell^{\infty}$ norm for the difference of matrices above.

By scaling and conjugating with a rotation, it suffices establish
\eqref{eq:reduction-1}
for $\xi = e_1$.
In this case, we use the angle addition formulas for $\cos(\theta)$ and $\sin(\theta)$,
to write
\begin{align*}
\begin{pmatrix}
\cos(\theta)^2 & \cos(\theta)\sin(\theta) \\
\cos(\theta)\sin(\theta) & \cos(\theta)^2
\end{pmatrix}
& =\frac12 \Id +
\frac12
\begin{pmatrix}
\cos(2\theta) & \sin(2\theta) \\
\sin(2\theta) & \cos(2\theta)
\end{pmatrix},
\end{align*}
which implies
\begin{equation}\label{eq:half-angle-reductio}
\left| \sum_{j=1}^N (\rho_{<j}^{-1}\xi)(\rho_{<j}^{-1} \xi)^\intercal - \frac12 N\Id\right|\leq C\left|\sum_{j=1}^N \rho_{<j}^{-2} \xi\right|^2
\end{equation}
To estimate the right hand side we consider the martingale $(Z_k)_{k\in [0,N]}$, given by
\[
Z_k := \Expec \left(\sum_{j=1}^N \rho_{<j}^{-2} \xi\  \Big| \mcal{F}_k\right).
\]
The key observation is that if
the support of $\mu$ has two elements with
different rotation parts, there exists some $c>0$,
such for any $\xi\in \mathbb{R}^2$,
with $|\xi| = 1$
\begin{equation}\label{eq:exp-decay}
|\Expec \rho_{\theta(g_1)}^{-1}\xi| < 1-c.
\end{equation}
We assume this is the case since otherwise the lemma
follows immediately as $\sum_{j=1}^N \rho_{<j}^{-2}\xi$ is
simply a geometric
series bounded and has magnitude at most $C$.
Iterating \eqref{eq:exp-decay} implies
\[
\left|Z_0\right| = \left|\Expec \sum_{j=1}^N \rho_{<j}^{-2} \xi\right| \leq
\sum_{j=1}^N \left|\Expec (\rho_{<j}^{-2} \xi)\right| \leq \sum_{j=1}^N e^{-cj} < C.
\]
Furthermore, \eqref{eq:exp-decay}  also implies that $(Z_i)_{i\in [0,N]}$ has bounded increments.
Indeed, we can estimate
\begin{align*}
|Z_{k+1} - Z_k|
&\leq
\sum_{j=1}^N
|\Expec( \rho_{<j}^{-2} \xi \mid \mcal{F}_{k+1}) - \Expec(\rho_{<j}^{-2} \xi | \mcal{F}_k)| \\
&\leq
\sum_{j=k}^N
|\Expec( \rho_{<j}^{-2} \xi \mid \mcal{F}_{k+1})| + |\Expec(\rho_{<j}^{-2} \xi | \mcal{F}_k)| \\
&\leq 2\sum_{j=k}^N e^{-c(j-k)} < 2C.
\end{align*}
Thus Azuma's inequality implies
$$\mathbb{P}(|\sum_{j=1}^N \rho_{<j}^{-2}\xi|\geq \lambda ) = \mathbb{P}(|Z_N|\geq \lambda )\leq Ce^{-\frac{\lambda^2}{N}},$$
which implies Lemma \ref{lem:markov-chain-estimate} via \eqref{eq:half-angle-reductio}
\end{proof}

Now we use Lemma \ref{lem:martingale-char-func-comp},
and then Lemmas \ref{lem:reduction-to-markov-chain} and Lemma \ref{lem:markov-chain-estimate},
to compute
\begin{align*}
\mathbb{E}e^{-i\langle Y_N, \xi\rangle}
& = e^{-\frac{1}{2}N\sigma^2 |\xi|^2 + O(N|\xi|^3)} \mathbb{E}e^{\frac{1}{2}N\sigma^2|\xi|^2 -\frac{1}{2}\sum_{j=1}^{N}\mathbb{E}\left[\langle \xi, w_j\rangle^2\ | \mathcal{F}_{j-1}\right]}\\
& =  e^{-\frac{1}{2}N\sigma^2 |\xi|^2 + O(N|\xi|^3+ N|\xi|^4)},
\end{align*}
which proves the proposition.

\section{Reduction to dense polynomial value property}
\label{sec:spectral}

In this section we explain how the ``dense polynomial value'' property
yields high frequency estimates on the characteristic function of $Y_N$.

Recall from  Definition~\ref{def:DPV} that we say $\theta\in \mathbb{R}$ satisfies the dense polynomial
value property $\DPV(n,R)$ if the set
\[
\left\{\log p(e^{i\theta}) \Big |\, p(z) = \sum_{j=0}^D c_j z^j \in \mathbb{Z}[z] \text{ with } \sum_{j=0}^D (1+|c_j|) < n\right\}
\]
is a $\frac{1}{10}$-net of $[-\log R, 0]\times [0,2\pi)$.
The following result bounds the decay of the characteristic function of $Y_N$
in terms of the $\DPV$ property.

\begin{proposition}
\label{prp:DPV-to-hf}
Let $\mu$ be a probability measure on $\Isom(\Real^2)$ with finite support
and $c>0$. If there exists $g_1,g_2\in \supp\mu$,  such
that $\theta(g_1g_2^{-1})$ satisfies $\DPV(n,R)$,
then for all $c<r<R$
the following estimate holds:
\[
\int_{\Sphere^1_r} |\Expec e^{i\langle \xi, Y_N\rangle}|^2 \diff \xi < \exp(-c'n^{-2}N).
\]
The constant $c'$ depends on $c$ and $\mu$.
\end{proposition}
There are two steps in the proof of Proposition~\ref{prp:DPV-to-hf}.  The first is based on an argument of
Varj\'u ~\cite{varju2015random} and Bourgain~\cite{bourgain2015}, which connects a spectral gap bound at high frequencies
to generating translations at many scales using short words.   The second, which is unique to the two-dimensional setting,
is to relate the translation elements in $\supp \mu^{\ast n}$ to polynomial values of $e^{i\theta}$.

First we explain the spectral gap in more detail.
Define the operator $T:\mcal{S}(\bbR^2)\to\mcal{S}(\bbR^2)$
by $$Tf(x) = \mathbb{E}_{g\sim \mu} f(gx).$$
Then for any $f\in \mathcal{S}(\mathbb{R}^2),$ and
$\xi\in \mathbb{R}^2$ with $|\xi| = r$ one has
\begin{equation}\label{eq:disintegration-along-spheres}
\begin{aligned}
\widehat{Tf}(\xi)
& = \mathbb{E}_{g\sim \mu}\int_{\mathbb{R}^2}f(g(x))e^{-i\langle x,\xi\rangle}dx\\
& = \mathbb{E}_{g\sim \mu} \int_{\mathbb{R}^2}f(x)e^{-i\langle g^{-1}(x),\xi\rangle}dx\\
& = \mathbb{E}_{g\sim \mu}\exp(i\langle \xi, v(g)\rangle) \hat{f}(\rho_{\theta(g)}(\xi)).
\end{aligned}
\end{equation}
Hence the decay of the fourier transform of $T^nf$ is related to the
spectral gap of the family of operators $(\hat{T}_r)_{r>0}$
acting on the spaces $(L^2(\mathbb{S}_r^1))_{r>0}$ via
\[
\hat{T}_r f := \Expec_{g\sim \mu} \hat{u}_r(g) f,
\]
with $\hat{u}_r$ being a unitary representation of $\Isom(\mathbb{R}^2)$ on $L^{2}(\mathbb{S}^1_r)$ given by
\begin{align}
\label{eq:ug-def}
\left(\hat{u}_r(g)f\right)(\xi) := \exp(i\langle \xi, v(g)\rangle) f(e^{i\theta(g)}\xi).
\end{align}
\begin{claim}
For any $r>0$
\begin{align}
\label{eq:ltwo-bd}
\int_{\Sphere^1_r} |\Expec e^{-i\langle \xi, Y_N\rangle}|^2 \diff \xi
= \|(\hat{T}_r)^N 1\|_{L^2(\Sphere^1_r)\to L^2(\Sphere^1_r)}^2.
\end{align}
\end{claim}
\begin{proof}
Since $\mathbb{E}f(Y_N) = T^Nf(0),$
the claim follows immediately from \eqref{eq:disintegration-along-spheres}.
\end{proof}

We will relate the spectral gap of $\hat{T}_r$ to our question about polynomials in two steps.
The first step
is to show that $1-\|\hat{T}_r\|$ is large if $\supp \mu^{\ast k}$ contains group elements that
perform pure translations at scale $r$ for some relatively small $k$.
The second step is to relate this set of pure translation elements to the values of polynomials.

\subsection{Bounding the spectral gap from ample translations}

We define $\trans_n(\mu)$ to be the
set of pure translations in $\supp \mu^{\ast n}$, i.e.
\[
\trans_n(\mu) := \{v\in\Real^2 \mid \tau_v \in \supp \mu^{\ast n}\}.
\]
When $\mu$ is not symmetric, it is possible that $\trans_n=\emptyset$.  In the nonsymmetric
case we will use instead $\trans_n(\mu\ast \mu^{-1})$.
We will need to show that the vectors in $\trans_n(\mu)$ (or $\trans_n(\mu\ast \mu^{-1})$)
are sufficiently dense at small scales.
The following definition encapsulates the notion of density we need.
\begin{definition}
A set $A\subset \Real^2$ is $\delta$-\emph{ample} for some $\delta>0$ if
the set $\{\log z \mid z\in A\}$ is an $0.1$-net in $[\log \delta, 0]\times [0,2\pi]$.
\end{definition}
Note that if for each $\delta < r < 1$
and $\theta \in [0,2\pi]$, there exists $v = \rho e^{i\phi} \in T$ such that
$r < \rho < 1.1r$ and $|\phi-\theta| < 0.1$, then $T$ is
$\delta$-\emph{ample}. When $\trans_n(\mu)$ is $\delta$-ample,
we have the spectral gap
$1-\|\hat{T}_r\|\gtrsim n^{-2}$ for all $r < \delta^{-1}$.

\begin{lemma}\label{lem:translations-imply-decay}
Fix any $c>0$. If $\mu=\mu^{-1}$ and
$\trans_n(\mu)$ is $\delta$-ample then for any
$c\leq r \leq \delta^{-1}$
we have
\[
\|\hat{T}_r\|_{L^2(\Sphere^1_r)\to L^2(\Sphere^1_r)}
\leq 1-\frac{c'}{n^2}.
\]
The constant $c'$ depends on $c$ and $\mu$. Alternatively if $\mu\not=\mu^{-1}$ and $\trans_n(\mu\ast \mu^{-1})$ is $\delta$-ample,
then for any $c\leq r \leq \delta^{-1}$ we have
\[
\|\hat{T}_r\|_{L^2(\Sphere^1_r)\to L^2(\Sphere^1_r)}
\leq \sqrt{1-\frac{c'}{n^2}}.
\]
The constant $c'$ depends on $c$ and $\mu$.
\end{lemma}

\begin{proof}
Suppose first that $\mu=\mu^{-1}$, so that $\hat{T}_r=\hat{T}_r^*$, and suppose that
\[
\|\hat{T}_r\|_{L^2(\Sphere^1_r)\to L^2(\Sphere^1_r)}= 1-\delta.
\]
Since $\hat{T}_r$ is self-adjoint, there exists $f\in L^2(\mathbb{S}^1_r)$, such that
\[
\|\hat{T}_r(f)-f\|_2\lsim \delta.
\]
We can write $\hat{T}_r$ as an average
over $\mu$ of the the action of the unitary representations $\hat{u}_r(g)$,
defined in \eqref{eq:ug-def}.
Since $\hat{T}_rf = \Expec_\mu \hat{u}_r(g) f$, $\mu$ has finite support, and $\|\hat{T}_r(f)-f\|_2\lsim \delta$,
it follows that
\[
\|\hat{u}_r(g) f-f\|_2 \lsim \delta^{1/2}
\]
for all $g\in \supp\mu$.
Again using unitarity, this implies that for any $g_n\in \supp(\mu^{\ast n})$, we have
\begin{equation}
\label{eq:iterated}
\|\hat{u}_{r}(g_n)f-f\|_2 \lsim n\delta^{1/2}.
\end{equation}

Since $\|f\|_2 = 1$, there exists $\xi_0\in \mathbb{S}^1_r$, such that
\[
\int_{\Sphere^1_r}\One_{|\xi-\xi_0|<r/10} |f(\xi)|^2 d\xi \gsim 1.
\]
On the other hand, since $\trans_n(\mu)$ is $\delta$-ample, and  $c\leq r \leq \delta^{-1}$,
there is some $v$ with $\tau_v\in\supp(\mu^{\ast n})$ such that
\[
\frac{c}{2}\leq \langle \xi_0, v\rangle \leq \frac12.
\]
But then using the definition~\eqref{eq:ug-def} of $\hat{u}_r(g)$ we have
\[
\|\hat{u}_r(g) f-f\| = \|e^{i\langle \cdot, v\rangle}f -f\|_2 \geq c'
\]
which combined with \eqref{eq:iterated} yields $n\delta^{1/2} \geq c'$, implying the claim.

If $\mu\not=\mu^{-1}$ we consider instead the operator $\hat{T}_r^*\hat{T}_r$ which is
$\hat{T}_r(\mu\ast \mu^{-1})$.  Applying the result from the symmetric case yields
\begin{align*}
\|\hat{T}_r f\|_{L^2}^2
\leq \|f\|_{L^2} \|\hat{T}_r^*\hat{T}_rf\|_{L^2}
\leq (1-\frac{c'}{n^2}) \|f\|_{L^2}^2.
\end{align*}
\end{proof}

\subsection{Relation to polynomial values}
For a polynomial $p(z)=\sum_{j=0}^d c_j z^j$ with integer coefficients we define the weight $w(p)$ to be
\[
w(p) := \sum_{j=0}^d (|c_j|+1) = \deg p + \|p\|_{\ell^1} + 1.
\]
Further for any $z\in \mathbb{C}$, we define the set
of polynomial values
$$V_n(z) := \left\{p(z) \mid w(p) \leq n\right\}.$$
The following lemma connects $V_n(z)$ to $\trans_n(\mu)$.

\begin{lemma}
\label{lem:poly-to-trans}
Suppose $\mu=\mu^{-1}$ is a probability measure with finite support on $\Isom(\Real^2)$ and let $g_1=\tau_{v_1} \circ \rho_{\theta_1} \in \supp \mu$ and
$g_2=\tau_{v_2} \circ \rho_{\theta_2}\in\supp \mu$ with $g_1g_2\not=g_2g_1$.
Then for some $a\in \bbC\setminus\{0\}$,
\[
a V_n(e^{i\theta_1})\subset \trans_{4n}(\mu).
\]
\end{lemma}
\begin{proof}
We argue by induction.  First observe that the commutator of $g_1$ and $g_2$ generates a pure translation:
\begin{equation}
\label{eq:gcommutator}
g_1 g_2 g_1^{-1} g_2^{-1} = \tau_a
\end{equation}
for some $a\in\bbC\setminus \{0\}$ (by hypothesis).
Also note that for any $u\in\bbC$,
\begin{equation}
\label{eq:u-conjugation}
g_1 \tau_u g_1^{-1} = \tau_{e^{i\theta_1} u}.
\end{equation}

In particular, the commutator identity~\eqref{eq:gcommutator} implies
\[
a V_2 \subset \trans_8,
\]
as $\pm1$ are the only two polynomials with $w(p) \leq 2$.
This is the start of our induction.  Suppose for some $k \geq 2$, we have
$a V_k \subset \trans_{4k}(\mu)$, and let $w(p) = k+1$.  Then for some $q$ with
$w(q) \leq k$ we have either $p=q\pm 1$ or else $p(x)=xq$.
In the former case, we have
\[
a p(z) = a q(z) + a,
\]
and we can express translation by $ap(z)$ by a word of length $4k+4$ using~\eqref{eq:gcommutator},
\[
\tau_{ap(z)} = \tau_{aq(z)} \tau_a
\]
In the other case $p(z)=zq(z)$ we can express $p(z)$ using a word of length $4k+2$
using~\eqref{eq:u-conjugation},
\[
\tau_{ap(z)} = \tau_{azp(z)} = g_1 \tau_{ap(z)} g_1^{-1}.
\]
\end{proof}

\subsection{Proof of Proposition~\ref{prp:DPV-to-hf}}
If $\theta$ satisfies $\DPV(n,R)$, then $V_n(e^{i\theta})$ is $R^{-1}$-ample.
Lemma~\ref{lem:poly-to-trans} then implies that $\trans_{4n}(\mu)$ is $R^{-1}$-ample,
so by Lemma~\ref{lem:translations-imply-decay} we have
\[
\|\hat{T}_r\|_{L^2(\Sphere^1_r)\to L^2(\Sphere^1_r)} \leq 1 - \frac{c'}{n^2}
\]
for all $c<r<R$. Then the conclusion of Proposition~\ref{prp:DPV-to-hf} follows from~\eqref{eq:ltwo-bd}.

\section{Random isometries with irrational rotations}
\label{sec:poly-reduction}
In this section we prove Proposition~\ref{prp:super-poly}, which states that the dense polynomial value property for irrational $\theta$
holds down to superpolynomial scales.  We then show how this implies Theorem~\ref{thm:main-diophantine}.

\subsection{Proof of Proposition~\ref{prp:super-poly}}
\label{sec:super-pf}
First we restate Proposition~\ref{prp:super-poly} for the reader's convenience.
\superpoly*

The proof is constructive, and the basic gadget we use in the construction is the polynomial
$p_{q,k} := (z^q-1)^k$. Note that $|p_{q,k}(z)| = |z^q-1|^k$ and
$w(p_{q,k}) = 2^k + qk + 1$.  Hence, if we can find $q < n^C$ so that $|z^q-1|\lsim n^{-1}$,
then taking $k \approx \log(r^{-1})/\log(n)$ implies $p_{q,k}$ is a polynomial of weight $n^C$ satisfying
$|p_{q,k}(z)| < r$.  The diophantine condition
can be used to give a rough lower bound on the size of $|p_{q,k}(z)|$
and by tweaking $p_{q,k}$ one can then adjust the angle and improve the precision of the absolute value.

\begin{lemma}
\label{lem:goodq}
Suppose $z=e^{i\theta}$ for $\theta$ satisfying~\eqref{eq:DCm}. There exists $C>0$
such that for any $n\in\bbN$, there exists $q\in\bbN$, $q\leq Cn^{m}$,
such that
\begin{equation}
\label{eq:goodq}
\frac{1}{2n} < |z^q-1| < \frac{1}{n}
\end{equation}
\end{lemma}
\begin{proof}
Taking $C$ large, and pigeonholing on the set $\{z^j\}_{j=1}^{Cn}$, we can find $1\leq j<j' \leq Cn$
such that $|z^{j'-j}-1| = |z^j-z^{j'}| < n^{-1}$. On the other hand,
the Diophantine condition implies $|z^{j'-j}-1| \geq c n^{-m}$.
Hence ~\eqref{eq:goodq} holds for some $q=\ell (j'-j)$ with
$\ell \leq Cn^{(m-1)}$.
\end{proof}

We are now ready to prove Proposition~\ref{prp:super-poly}.
\begin{proof}[Proof of Proposition~\ref{prp:super-poly}]
Fix a large $n$, a scale
$e^{-c\log(n)^2}\leq r\leq 1$,
and an angle $\varphi\in [0,2\pi)$.
Note that by adjusting
$c$ in the statement and taking $n$ large
it suffices to prove
there exists $p(z)\in \mathbb{Z}[z]$
such that $w(p)\leq n^C$,
$r\leq |p(z)|\leq 1.1r,$ and
$|\arg(p(z))-\varphi|\leq 0.1$,
for some constant $C$ depending only on $\theta$.

Our $p(z)$ will be of the form
\[
p(z) = \ell z^{j}(z^q-1)^k,
\]
for some $\ell, j, q,k\in \mathbb{N}$.
The integers $q$ and $k$ are chosen
so that $(z^q-1)^k$ is the roughly the right magnitude,
and the factors $\ell$ and $z^j$ are chosen to
correct the magnitude and angle.
Note that $p$ has weight
\begin{equation}\label{eq:norm-estimate}
w(p) = \deg p + \|p\|_{\ell^1} + 1
= j + qk + \ell 2^k + 1.
\end{equation}

We choose first $q$ using Lemma~\ref{lem:goodq} such that $q\leq n^{m+1}$
and
\[
\frac{1}{2n} \leq |z^q-1| \leq \frac1n.
\]
Then set $k = \lceil \frac{\log (0.01 r)}{\log |z^q-1|}\rceil$, which satisfies $2^k\leq n^C$
and
\[
\frac{r}{200n} \leq |z^q-1|^k \leq \frac{r}{100}.
\]
Then choose $\ell\leq 200n$ such that
\[
r \leq \ell |z^q-1|^k \leq 1.1r.
\]
Finally, since $\theta$ is irrational there is $j\leq C$, such that
\[
|\arg(z^j\ell(z^q-1)^k)-\varphi|\leq \frac{1}{10}.
\]
Our bounds on $\ell$, $j$, $k$, and $q$ together with~\eqref{eq:norm-estimate}
imply that $w(p) \leq n^C$, as desired.
\end{proof}

\subsection{Proof of Theorem~\ref{thm:main-diophantine}}
\label{sec:diophantine-pf}
By Lemma~\ref{lem:llerr}, it suffices to prove that
\begin{align}
\label{eq:llt-ft}
\int_{|\xi|<\exp(c (\log n)^2)} |\Expec e^{-i\langle \xi,Y_N\rangle} - \Expec e^{-\frac12 \sigma^2N|\xi|^2}|d\xi
\lsim N^{-\frac32}.
\end{align}
For $c>0$ small, the low frequency estimate of Proposition~\ref{prp:lowfreq} implies
\begin{align*}
\int_{|\xi|\leq c} |\Expec e^{-i\langle \xi,Y_N\rangle} - \Expec e^{-\frac12 \sigma^2N|\xi|^2}|d\xi
& \lsim \int_{|\xi|\leq c} e^{-\frac{1}{2}\sigma N|\xi|^2}|1-e^{CN|\xi|^3}|d\xi \\
& \lsim N\int_{|\xi|\leq c} e^{-\frac{1}{2}\sigma N|\xi|^2}|\xi|^3 e^{CN|\xi|^3}d\xi \\
& \lsim N^{-3/2}.
\end{align*}
It therefore suffices to show that
\begin{align}
\label{eq:hf-FT}
\int_{c<|\xi|<\exp(c (\log n)^2)} |\Expec e^{-i\langle \xi,Y_N\rangle}| \diff \xi \lsim N^{-\frac32}.
\end{align}

We split the argument into two cases.  The first case is that $\theta(g_1^{-1}g_2)$ satisfies~\eqref{eq:DCm}
for some $g_1,g_2\in\supp\mu$. In this case, Proposition~\ref{prp:super-poly} combined with Proposition~\ref{prp:DPV-to-hf} implies that for $r < \exp(c (\log N)^2)$,
\[
\int_{\Sphere^1_r} |\Expec  e^{i\langle \xi, Y_N\rangle}|^2 \diff \xi < \exp(-c N^{0.1}).
\]
Using Cauchy-Schwarz and integrating in radial coordinates, this implies~\eqref{eq:hf-FT}.

The second case is that $\theta(g) = \theta_0$ for all $g\in\supp\mu$.  In this case, the characteristic function of
$Y_N$ can be written explicitly as follows:
\[
\Expec e^{i\langle \xi, Y_N\rangle}
= \prod_{j=0}^{N-1} \hat{\nu}(e^{ij\theta_0}\xi),
\]
where $\nu \in \mcal{M}(\Real^2)$
is the probability measure of $g(0) \overset{d}{=} \tau_g$.  Above we
have identified $\Real^2$ with $\bbC$, meaning that we write
$\hat{\nu}(\xi)$ for the two-dimensional Fourier transform of $\nu$, evaluated at
$\xi\in\Real^2\simeq \bbC$.
Up to a global rotation and scaling
we can assume that $\supp(\Rept \mu)$ contains two atoms seperated by $\pi$,
so that $\Ft{\nu}$ satisfies a bound of the form
\[
|\hat{\nu}(\xi)| \leq 1 - cd(\Rept \xi, \bbZ)^2 < \exp(-cd(\Rept \xi,\bbZ)^2).
\]
Let $c < |\xi| < \exp(c(\log N)^2)$.  We will prove that for any $k\in[N]$,
\begin{align}
\label{eq:tokyo}
\sum_{j=k}^{k+N^{0.1}} d(\Rept [e^{ij\theta_0} \xi],\bbZ)^2 \gtrsim N^{-0.2},
\end{align}
which implies $|\Expec e^{i\langle \xi, Y_N\rangle}| < \exp(-cN^{0.7})$, and therefore~\eqref{eq:hf-FT}.
Now we prove~\eqref{eq:tokyo}.
If we define $\xi' := e^{ik\theta_0}\xi$, then
Proposition~\ref{prp:super-poly} implies that there
exists a polynomial $p$ such that
$w(p) < N^{0.1}$ and $0.001 < \Rept(p(e^{i\theta_0}) \xi') < 0.1$.  Then, with $c_j$ being the coefficients of $p$,
\begin{align*}
c' &<
d (\Rept [p(e^{i\theta_0}) \xi'],\bbZ) \\
&< \sum_{j=0}^{N^{0.1}} |c_j| d(\Rept [e^{ij\theta_0}\xi'], \bbZ) \\
&\leq \|p\|_{\ell^2} \Big(\sum_{j=k}^{k+N^{0.1}} d(\Rept [e^{ij\theta_0}\xi],\bbZ)^2\Big)^{1/2}.
\end{align*}
Now~\eqref{eq:tokyo} follows from $\|p\|_{\ell^2} \leq \|p\|_{\ell^1} < w(p)$.

\section{Symmetric random isometries with rational cosine}
\label{sec:symmetric}
In this section, we prove Proposition~\ref{prp:exp-Rbds} and complete the proof of Theorem~\ref{thm:main-symmetric}.
restated below.  These concern the case that $\mu$ has a rotation with angle $\theta$ satisfying
$\cos\theta \in\bbQ\setminus \frac12\bbZ$.

In general, if $\cos(\theta)\in\bbQ$ then $z=e^{i\theta}$ satisfies
\begin{align}
\label{eq:minimal-poly} p_{\rm min}(z) = a z^2 - b z + a = 0,
\end{align}
with $a,b\in\bbZ$ coprime and $\cos(\theta)=\frac{b}{2a}$.
Since $\cos(\theta)\not\in\frac12\bbZ$
we can rule out the case $a=1$, so we can assume $a>1$ and we have $|b|<2a$.  In this case,
$p_{\rm min}(z)$ is indeed the minimal polynomial of $z$.  We will fix $\theta, p_{\rm min}$, $a$,
and $b$ throughout this section.

First in Section~\ref{sec:matveev} we state a Diophantine condition for $\theta$.
Then in Section~\ref{sec:resultant} we prove that $z$ satisfies the ``dense polynomial
values'' condition, establishing the bound $R_{\DPV}(n,\theta) > \exp(cn/\log(n))$.
Finally, in Section~\ref{sec:sym-proof} we complete the proof of Theorem~\ref{thm:main-symmetric}.

\subsection{Diophantine approximations to $\theta$}
\label{sec:matveev}
The point of this section is to establish that $\theta$ satisfies $\eqref{eq:DCm}$.
\begin{proposition}
\label{cor:matv}
Let $z =  e^{i\theta}$ with $\cos\theta\in \mathbb{Q}\backslash\frac{1}{2}\mathbb{Z}$.
Then there exist constants $c>0$ and $m\in \mathbb{N}$ such that for $p,q,k\in\bbZ$,
\begin{align}
\label{eq:diophantine}
|\theta - \frac{2\pi p}{q}| > c q^{-m}
\end{align}
and also
\begin{align}
\label{eq:rept-lower}
|\Rept z^k| > c k^{-m}.
\end{align}
\end{proposition}

While it is easily seen that $\theta$ must be irrational (if it were rational, then
$p_{min}(z)$ in \eqref{eq:minimal-poly} would divide $z^{q}-1$, for some $q$,
which is impossible since $a>1$), it is not obvious that $\theta$
should satisfy \eqref{eq:DCm}.  This more quantitative result is a consequence of
Baker's celebrated work on linear forms in logarithms of algebraic numbers~\cite{MR220680,MR498417,MR2382891}.

\begin{theorem}[Special case of~\cite{MR2382891}, Theorem 7.1]
\label{thm:matveev}
Let $\alpha,\beta\in\bbC$ be algebraic numbers and $\log\alpha,\log\beta$ be logarithms of
$\alpha$ and $\beta$.  If $\log\alpha$ and $\log\beta$ are independent
over $\mathbb{Z}$, then there exist constants $c,C>0$, depending on $\alpha$ and $\beta$,
such that for any $a_1,a_2\in\bbZ$,
\[
|a_1 \log \alpha + a_2 \log\beta| \geq c (\max\{|a_1|,|a_2|\})^{-C}.
\]
\end{theorem}

Since $\theta$ is irrational,
Proposition \ref{cor:matv} is an immediate corrollary of the above theorem.

\begin{proof}[Proof of Proposition \ref{cor:matv}]
Equation~\eqref{eq:diophantine} follows from applying Theorem~\ref{thm:matveev} with
$\alpha=z$, $\log\alpha=i\theta$ and $\beta=1$, $\log\beta=2\pi i$ to obtain
\[
|q \theta - 2\pi p| > c (\max\{|p|,|q|\})^{-C},
\]
and observing we can assume $|p|<|q|$.
The second bound~\eqref{eq:rept-lower} follows from applying Theorem~\ref{thm:matveev}
with $\alpha=z$ and $\beta=i$ (equivalently, it also follows from~\eqref{eq:diophantine}
directly).
\end{proof}

\subsection{Dense polynomial values for rational cosines}
\label{sec:resultant}
We are ready to prove Proposition~\ref{prp:exp-Rbds}, restated below.
\algebraic*

As opposed to the proof of Proposition \ref{prp:super-poly}, the proof of
the above is non-constructive. Namely, we apply the pigeonhole principle
to the following set of polynomials:
\[
\mcal{P}_D := \{q \in\bbZ[x] \mid q(x) = \sum_{j=0}^D b_j x^j, 0\leq b_j < a\}.
\]
Note each polynomial in $\mathcal{P}_D$ has weight $O(D)$ ($a$ is fixed as the leading coefficient in $p_{\rm min}$).
The pigeonhole principle will guarantee the existence of $q_1,q_2\in\mcal{P}_D$, such that
$|q_1(z)-q_2(z)| \lsim a^{-D/2}$, and then we establish a matching
lower bound by using the resultant $\Rs(p_{\rm min}, q_1-q_2)$.

The first step is the pigeonhole principle.
\begin{lemma}
\label{lem:pigeon}
The map $\iota_z:\mcal{P}_{D}\to \bbC$ defined by evaluation $q\mapsto q(z)$ is
injective for $z$ having minimal polynomial $p_{\rm min}$
with leading coefficient $a$.   As a consequence,
for any $D>1$ there exists a polynomial $q$ of degree $D$ and integer coefficients bounded
by $a-1$ such that $|q(z)| \lsim D\cdot a^{1-D/2}$.
\end{lemma}
\begin{proof}
Suppose $q_1\not=q_2\in\mcal{P}_{D}$.
Then $q_1-q_2$ has coefficients in $[1-a,a-1]$, so the
leading coefficient is not divisible by $a$.
Therefore $p_{\rm min}$ cannot divide $q_1-q_2$, so $q_1(z)\not=q_2(z)$.

The second statement then follows from the injectivity
above, the pigeonhole principle, and the fact that
$|q(z)| \leq a(D+1)$ for any $q\in\mcal{P}_{D}$.
\end{proof}

To lower bound $|q(z)|$ we use the \textit{resultant}.
The resultant of two polynomials $p$ and $q$ with integer coefficients is an integer which vanishes only if $p$ and $q$ share a root.
Suppose $p$ has degree $D$, leading coefficient $t$, and zeros $\{\lambda_j\}_{j=1}^D$ (counted with multiplicity), and
$q$ has degree $F$, leading coefficient $u$, and zeros $\{\mu_j\}_{j=1}^{F}$.  Then $\Rs(p,q)$ is given by
\[
\Rs(p,q) = t^{F} u^D \prod_{j\in[D],k\in[F]} (\lambda_j-\mu_k) = t^F\prod_{j\in[D]}q(\lambda_j) = (-1)^{DF} u^D \prod_{k\in[F]}p(\mu_k).
\]
$\Rs(p,q)$ is an integer, since it can be written
as a the determinant of an integer matrix.
In particular, $\Rs(p,q)\not=0$ implies $|\Rs(p,q)|\geq 1$.
This immediately implies a lower bound for $q(z)$.

\begin{lemma}
\label{lem:resultant-lowerbd}
For any $q\in \mcal{P}_D$
\[
|q(z)| \geq a^{-D/2}.
\]
\end{lemma}
\begin{proof}
The resultant $\Rs(p_{\rm min},q)$ is nonzero (because $p_{\rm min}$ cannot divide $q$),
so we have
\[
1 \leq |\Rs(p_{\rm min},q)| = a^D |q(z)|^2.
\]
\end{proof}

We can now prove Proposition~\ref{prp:exp-Rbds}.
\begin{proof}[Proof of Proposition~\ref{prp:exp-Rbds}]
Combining Lemma~\ref{lem:pigeon} with Lemma~\ref{lem:resultant-lowerbd} we can find for each
$d$ a polynomial $q$ with weight $w(q) < (2a+1)d$ such that
\[
a^{-d/2} \leq |q_d(z)| \leq C d a^{-d/2}.
\]
Given $r > \exp(-cn/(\log n)^2)$ we can find $q_1$ with weight $w(q_1) < \frac{cn}{\log n}$
and degree $d_1\leq\frac{cn}{(\log n)^2}$ such that
\[
r < |q_1(z)| < C d_1 r.
\]
Iterating this twice more we can find polynomials $q_2$ and $q_3$ such that
$w(q_2) < C\log n$ and $w(q_3) < C\log\log n$ and
\[
r < |q_1(z)q_2(z)q_3(z)| < C (\log\log n)r.
\]
Now by applying Proposition~\ref{prp:super-poly} (which is applicable by Corollary~\ref{cor:matv}),
we can find $q_4$ with $w(q_4) < C(\log\log n)^C$
\[
r < |q_1(z)q_2(z)q_3(z)q_4(z)| < 1.1r.
\]
Note that for $n$ large, $w(q_1q_2q_3q_4) < C\frac{n}{(\log n)^2} (\log n) (\log\log n)^{C+1} < n$.
Finally, since $\theta$ is irrational, we can
adjust the angle if necessary by multiplying by $z^k$ for bounded $k$.
\end{proof}

\subsection{Proof of Theorem~\ref{thm:main-symmetric}}
\label{sec:sym-proof}
We show
$$\int_{|\xi|<\exp(c N^{1/3}/(\log N)^2)} |\Expec e^{-i\langle \xi,Y_N\rangle} - \Expec e^{-\frac12 \sigma^2N|\xi|^2}|d\xi\lsim N^{-3/2}$$
just as in Section~\ref{sec:diophantine-pf}. The low frequency contribution,
 $|\xi|\leq c$, can be estimated the same way, and thus it suffices to show that, for some $c>0$,
\begin{align}
\label{eq:cos-FT}
\int_{c<|\xi|<\exp(c N^{1/3}/(\log N)^2)} |\Expec e^{-i\langle \xi,Y_N\rangle}| \diff \xi \lsim N^{-\frac32}.
\end{align}
Combining Proposition~\ref{prp:exp-Rbds} with Proposition~\ref{prp:DPV-to-hf} implies
that for $c<r < \exp(c n / (\log n)^2)$
\[
\int_{\Sphere^1_r} |\Expec e^{-i\langle \xi,Y_N\rangle}|^2 \diff \xi \lsim \exp(-c n^{-2} N).
\]
Taking $n = N^{1/3}$, we conclude that for $r < \exp(c N^{1/3} /(\log N)^2)$, we have
\[
\int_{\Sphere^1_r} |\Expec e^{-i\langle \xi,Y_N\rangle}| \diff \xi \lsim \exp(-c N^{1/3}),
\]
from which~\eqref{eq:cos-FT} follows.

\section{Asymmetric random isometries with rational cosines}
\label{sec:aadic}
In this section we prove Theorem~\ref{thm:main-asymmetric}, which treats the case that $\mu$ is an ``asymmetric'' random isometry with an angle $\theta$ having rational cosine $\cos\theta\in\bbQ\setminus \frac12\bbZ$.

The main ingredient is the following proposition, which states that $d(r\cos(j\theta+\varphi),\bbZ)$ cannot be
too small for many consecutive values of $j$.
\begin{proposition}
\label{prp:main-five}
Let $\theta\in[0,2\pi)$ such that $\cos\theta\in\bbQ\setminus \frac12\bbZ$.
Then there exist $0<c<C$ and such that
for any $n\in \mathbb{N}$ sufficiently large, $n^{C}<r<\exp(cn)$, $\varphi\in[0,2\pi)$,
and $j\in[n]$, there exists $j'\in[j,j+C\log r]$ such that
\[
d(r \cos(j'\theta+\varphi),\bbZ) > c.
\]
\end{proposition}
The idea behind Proposition~\ref{prp:main-five} is that $r\cos(j\theta+\varphi)$ satisfies a recurrence relation with a
transfer matrix that is hyperbolic in a $p$-adic sense.
Proposition~\ref{prp:main-five} is proven in Section~\ref{sec:five-adic-pf} below, and then Theorem~\ref{thm:main-asymmetric} is deduced from Proposition~\ref{prp:main-five} in Section~\ref{sec:asym-pf}.

\subsection{Divisibility properties of $r\cos(j\theta+\varphi)$}
\label{sec:five-adic-pf}
Fix $\theta$, $n$, $r$, and $\varphi$ as in Proposition~\ref{prp:main-five}.  We set $a,b\in\bbZ$ such that
$\cos(\theta) = \frac{b}{2a}$, exactly as in~\eqref{eq:minimal-poly}.
The identity
$\cos(\alpha+\theta) + \cos(\alpha-\theta) = 2\cos(\theta)\cos(\alpha)$,
implies that  $y_j := r\cos(j\theta +\varphi)$ satisfies the following recurrence
\begin{align}
\label{eq:yj-recurrence}
y_{j+1} + y_{j-1} = \frac{b}{a}y_j.
\end{align}
We set $m_j$ to be the nearest integer to $y_j$, and $t_j$ to be the rounding error,
so that
\[
d(y_j,\bbZ) = |y_j - m_j| = t_j.
\]

The key observation is that $m_j$ also satisfies the recurrence~\eqref{eq:yj-recurrence} when $y_j$ are sufficiently
well approximated by integers.
\begin{lemma}
\label{lem:m-recurrence}
Suppose $1\leq j\leq n$ and $\max\{|t_{j-1}|,|t_j|,|t_{j+1}|\} < \frac{1}{4ab}$.
Then the integer approximants $m_j$ satisfy the recurrence
\begin{equation}
\label{eq:mj-recurrence}
m_{j+1} + m_{j-1} = \frac{b}{a} m_{j}.
\end{equation}
\end{lemma}
\begin{proof}
Multiply the recurrence~\eqref{eq:yj-recurrence} by $a$, writing
$y_j = m_j + t_j$ to obtain
\[
a m_{j+1} + am_{j-1} + (a t_{j+1} + a t_{j-1}) = b m_j + ab t_j.
\]
Rearranging gives
\[
a m_{j+1} - bm_j + am_{j-1} = ab t_j - a t_{j+1} - a t_{j-1}.
\]
The bound on $t_{j-1}, t_j, t_{j+1}$ implies that the right hand side is bounded by $\frac34$, but on the other hand must also be an integer,
so it vanishes.  Rearranging completes the proof of the claim.
\end{proof}

For any prime $p$ and $k\in \mathbb{Z}$, let $v_p(k)$ be the $p$-adic evaluation of $k$,
given by
$$v_p(k) :=
\begin{cases}
\max\{j\in \mathbb{Z}\mid p^j \text{ divides }k\}, &\text{ if }k\neq 0 \\
\infty, &\text{ if } k =0.
\end{cases}$$
As a consequence of~\eqref{eq:mj-recurrence}, if $p$ divides $a$, then
\[
v_p(m_j) \geq \min\{v_p(m_{j+1}),v_p(m_{j-1})\} + v_p(a).
\]
Since $a$ and $b$ are co-prime, the relation \eqref{eq:mj-recurrence} implies by
induction the following corollary.
\begin{corollary}
\label{cor:p-adic}
Suppose $1\leq j\leq n$ and $\max\{|t_{j-k}|\}_{|k|\leq K} < \frac{1}{4ab}$ and let $p$ be a prime dividing $a$.
Then $v_p(m_j) \geq K v_p(a)$.
\end{corollary}

The proof of Proposition~\ref{prp:main-five} immediately follows.
\begin{proof}[Proof of Proposition~\ref{prp:main-five}]
Let $r,\varphi,\theta$, and $n$ as in the statement of Proposition~\ref{prp:main-five}.
As above, let $y_j := r\cos(j\theta + \varphi)$ and $m_j$ be the nearest integer to $y_j$.
Note that since $r>n^C$, $j\leq n$ and $n$ is sufficiently large, Corollary~\ref{cor:matv}
implies that $m_j\not=0$.

Let $p$ be a prime dividing $a$ and suppose $|y_{j'}-m_{j'}| < \frac{1}{4ab}$ for all $|j'-j_0| \leq 2\log_p r$.
Then Corollary~\ref{cor:p-adic} implies that $p^{\lceil 2\log_p r\rceil} | m_j$,
and thus $|m_j|\geq r^2$, since $m_j\neq 0$.
But by definition $|m_j-r\cos(j\theta +\phi)|\leq \frac{1}{2}$
for all $j$, so $|m_j|\leq r$, which is a contradiction.
Therefore there must exist $j'$ with $|j'-j_0| < \log_p r$, and
$d(y_{j'},\bbZ) > \frac{1}{4ab}$, as desired.
\end{proof}

\subsection{Proof of Theorem~\ref{thm:main-asymmetric}}
\label{sec:asym-pf}
We treat the two cases of the Littlewood polynomials and the asymmetric wreath product separately.

In the Littlewood polynomial case, we have for $\xi=re^{i\varphi}$ the identity
\[
\Expec e^{i\langle \xi,Y_N\rangle}
= \prod_{j=0}^{N-1} \cos(r\cos(j\theta+\varphi)).
\]
Now suppose $N^C < r < \exp(c\sqrt{N})$.  Then by Proposition~\ref{prp:main-five}, for each $j\in[N]$
there is some $j'\in[j,j+\sqrt{N}]$ such that $\cos(r\cos(j'\theta+\varphi)) < 1-c$.  Therefore,
\[
|\Expec e^{i\langle \xi,Y_N\rangle}| \leq \exp(-c\sqrt{N}).
\]
This implies that
\[
\int_{N^C < |\xi| < \exp(c'\sqrt{N})} |\Expec e^{i\langle \xi,Y_N\rangle}| \diff\xi
\lsim N^{-\frac32},
\]
where we decreased $c'$ as necessary. Combined with the bound~\eqref{eq:hf-FT} (which applies because, by Corollary~\ref{cor:matv}, $\theta$
satisfies a Diophantine condition), this implies~\eqref{eq:llt-ft} and therefore the result.

The second case is that of the asymmetric wreath product.  In this case there is also a formula
for the characteristic function of $Y_N$, but it is more complicated.  Let $\xi = re^{i\varphi}$, and let
$\alpha_n := \theta(g_1\cdots g_n)$ be the rotational part of $g_1\cdots g_n$.  Note that $\alpha_n = j_n \theta$
for some $j_n\in\bbZ$.  Let $\sigma_n$ be the variable
\[
\sigma_n := \begin{cases}
1 & g_n = \tau_1 \\
-1 & g_n = \tau_{-1} \\
0 & g_n \in\{\rho_{\pm\theta}\}.
\end{cases}
\]
Then
\[
Y_N = \sum_{n=1}^N \sigma_{n} e^{i\alpha_{n-1}},
\]
so conditional on the values of the angles $\alpha_n$ and writing $A_n := \sigma_n^2 = \One(g_n \in\{\tau_{\pm1}\})$, we have
\[
\Expec (e^{i\langle \xi, Y_{N}\rangle} | (\alpha_n)_{n=1}^N)
= \prod_{\substack{n\in [N] \\ A_n=1}} \cos(r\cos(\alpha_{n-1}+\varphi))
\]
Therefore,
\[
|\Expec e^{i\langle \xi, Y_{N}\rangle}|
\leq \prod_{j=0}^{\sqrt{N}}
|\Expec (\prod_{\substack{n\in[j\sqrt{N},(j+1)\sqrt{N}] \\ A_n=1}} \cos(r\cos(\alpha_{n-1}+\varphi)) | \mcal{F}_{j\sqrt{N}})|,
\]
where $\mcal{F}_n$ is the sigma algebra generated by $g_1,\dots,g_n$.
In particular, the bound
\begin{align}
\label{eq:asym-wreath-ft}
|\Expec e^{i\langle \xi, Y_{N}\rangle}| \leq \exp(-c\sqrt{N})
\end{align}
would follow from the estimate
\begin{align}
\label{eq:uniform-prob-bd}
\Prob(A_n =1\text{ for some } n < \sqrt{N}\text{ and }\cos(r\cos(\alpha_{n-1}+\varphi)) < 1-c) > c
\end{align}
with constants uniform on $\varphi$.  On the other hand, for $r < \exp(c\sqrt{N})$ and
any $\varphi$ Proposition~\ref{prp:main-five} implies that there exists some $|j|\leq c\sqrt{N}$ such that
\[
\cos(r\cos(j\theta+\varphi)) < 1-c.
\]
We have that $\alpha_n = j_n\theta$ for integers $j_n\in\bbZ$ performing a random walk with a nonzero drift, so that
for small $c$,
\[
\Prob(\text{there exists }n\leq \sqrt{N} \text{ such that }\cos(r\cos(\alpha_n+\varphi)) < 1-c) > c.
\]
For such $n$, $\Prob(A_{n+1}=1|\mcal{F}_n) > c$.  Thus~\eqref{eq:uniform-prob-bd} follows, and therefore the
estimate~\eqref{eq:asym-wreath-ft}.

\appendix
\section{Local limit theorem from Fourier bound}
\label{sec:llerr}
In this appendix we record a standard argument for recovering the local limit statement~\eqref{eq:local-lim}
from a bound on the Fourier transform.  We state a simplified variant of the result we need which is not at all sharp.
\begin{lemma}
\label{lem:llerr}
Let $\mu$ and $\nu$ be probability measures satisfying
\[
\int_{|\xi| < L} |\hat{\mu}(\xi) - \hat{\nu}(\xi)| \diff \xi < \eps,
\]
and additionally suppose that $\nu$ has density bounded by $\kappa$.
Then for any $r > L^{-1/4}$ and $x_0\in\Real^2$,
\begin{align}
\label{eq:ll-err}
|\Prob_{X\sim \mu} (X\in B_r(x_0)) - \Prob_{Y\sim \nu}(Y\in B_r(x_0))| \lsim \eps r^2 + L^{-3} + \kappa L^{-1/2} r^2.
\end{align}
\end{lemma}
\begin{proof}
Fix $r>L^{-1/4}$ and set $\delta = L^{-1/2}$.  As the hypotheses on $\mu$ and $\nu$ are translation invariant,
we may assume $x_0=0$.
Choose a smooth function $\rho_{r,\delta}\in C_c^\infty(\Real^2)$ satisfying
\begin{align*}
0 \leq \rho^-_{r,\delta} < \One_{B_r} < \rho^+_{r,\delta} &\leq 1 \\
\int |\rho^+_{r,\delta}(x) - \rho^-_{r,\delta}(x)|\diff x &\lsim \delta r  \\
|\widehat{\rho^\pm_{r,\delta}}(\xi)| &\lsim \min\{Cr^2, (c\delta |\xi|)^{-10}\}.
\end{align*}
We compute
\begin{align*}
\Big|\int \rho^\pm_{r,\delta}(x)(\diff \mu(x) - \diff\nu(x)) \Big|
&= \Big|\int \widehat{\rho^{\pm}_{r,\delta}} (\hat{\mu}(\xi) - \hat{\nu}(\xi)) \diff \xi\Big| \\
&\lsim \eps r^2  + \int_{|\xi| > L} |\widehat{\rho^\pm_{r,\delta}}(\xi)| \diff \xi \\
&\lsim \eps r^2 + L^{-3},
\end{align*}
where in the last line we have simplified using $\delta = L^{-1/2}$.
Then $\|\rho^\pm_{r,\delta} -\One_{B_r}\|_{L^1} \lsim \delta r$, so by the density bound for $\nu$ we have
\[
\Expec_{Y\sim \nu}|\One_{B_r}(Y) - \rho^\pm_{r,\delta}(Y)| \lsim \kappa L^{-1/2} r.
\]
Therefore we get an upper bound
\begin{align*}
\Prob_{X\sim\mu} \One_{B_r}(X)
&\leq
\Expec_{X\sim\mu} \rho^+_{r,\delta}(X) \\
&\leq
\Expec_{Y\sim\nu} \rho^+_{r,\delta}(Y) + C\eps r^2 + L^{-3}  \\
&\leq
\Expec_{Y\sim\nu} \One_{B_r}(Y) + C\eps r^2 + CL^{-3} +  C \kappa L^{-1/2} r.
\end{align*}
We obtain a matching lower bound by comparing to $\rho^-$, and thereby obtain~\eqref{eq:ll-err}.
\end{proof}

\printbibliography

@article{shalom-tao,
  title={A finitary version of Gromov’s polynomial growth theorem},
  author={Shalom, Yehuda and Tao, Terence},
  journal={Geometric and Functional Analysis},
  volume={20},
  number={6},
  pages={1502--1547},
  year={2010},
  publisher={Springer}
}

@incollection {guivarch,
    AUTHOR = {Guivarc'h, Yves},
     TITLE = {Equir\'epartition dans les espaces homog\`enes},
 BOOKTITLE = {Th\'eorie ergodique ({A}ctes {J}ourn\'ees {E}rgodiques,
              {R}ennes, 1973/1974)},
    SERIES = {Lecture Notes in Math.},
    VOLUME = {Vol. 532},
     PAGES = {131--142},
 PUBLISHER = {Springer, Berlin-New York},
      YEAR = {1976},
}

@article{bourgain2015,
    AUTHOR = {Bourgain, Jean},
     TITLE = {On random walks in large compact {L}ie groups},
 BOOKTITLE = {Geometric aspects of functional analysis},
    SERIES = {Lecture Notes in Math.},
    VOLUME = {2169},
     PAGES = {55--63},
 PUBLISHER = {Springer, Cham},
      YEAR = {2017},
}

@article{dawson2006solovay,
  title={The Solovay-Kitaev algorithm},
  author={Dawson, CM and Nielsen, MA},
  journal={Quantum Information and Computation},
  volume={6},
  number={1},
  pages={81--95},
  year={2006},
  publisher={Rinton Press}
}

@article{kitaev1997quantum,
  title={Quantum computations: algorithms and error correction},
  author={Kitaev, A Yu},
  journal={Russian Mathematical Surveys},
  volume={52},
  number={6},
  pages={1191},
  year={1997},
  publisher={IOP Publishing}
}

@article{varju2015random,
  title={Random walks in Euclidean space},
  author={Varj{\'u}, P{\'e}ter P{\'a}l},
  journal={Annals of Mathematics},
  pages={243--301},
  year={2015},
  publisher={JSTOR}
}

@article{lindenstrauss2016random,
  title={Random walks in the group of Euclidean isometries and self-similar measures},
  author={Lindenstrauss, Elon and Varj{\'u}, P{\'e}ter P},
  journal={Duke Mathematical Journal},
  volume={165},
  number={6},
  pages={1061--1127},
  year={2016},
  publisher={Duke University Press}
}

@article{bourgain2010spectral,
  title={Spectral gaps in $\mathit{su}(d)$},
  author={Bourgain, Jean and Gamburd, Alexander},
  journal={Comptes Rendus. Math{\'e}matique},
  volume={348},
  number={11-12},
  pages={609--611},
  year={2010}
}

@article{pittet2002random,
  title={On random walks on wreath products},
  author={Pittet, Christophe and Saloff-Coste, Laurent},
  journal={Annals of probability},
  pages={948--977},
  year={2002},
  publisher={JSTOR}
}

@article {KazdanUniform,
    AUTHOR = {Ka\v zdan, D. A.},
     TITLE = {Uniform distribution on a plane},
   JOURNAL = {Trudy Moskov. Mat. Ob\v s\v c.},
  FJOURNAL = {Trudy Moskovskogo Matemati\v ceskogo Ob\v s\v cestva},
    VOLUME = {14},
      YEAR = {1965},
     PAGES = {299--305},
   MRCLASS = {22.60 (10.33)},
  MRNUMBER = {193187},
}

@article {KhokhlovLocalLimit,
    AUTHOR = {Khokhlov, Yu.\ S.},
     TITLE = {A local limit theorem for the composition of random motions of {E}uclidean space},
   JOURNAL = {Dokl. Akad. Nauk SSSR},
  FJOURNAL = {Doklady Akademii Nauk SSSR},
    VOLUME = {260},
      YEAR = {1981},
    NUMBER = {2},
     PAGES = {295--299},
   MRCLASS = {60B15 (60F05)},
  MRNUMBER = {630143},
MRREVIEWER = {V.\ M.\ Maksimov},
}

@article{kestenbanach,
  title={Full Banach mean values on countable groups},
  author={Kesten, Harry},
  journal={Mathematica Scandinavica},
  pages={146--156},
  year={1959},
  publisher={JSTOR}
}

@article {kestensymmetric,
    AUTHOR = {Kesten, Harry},
     TITLE = {Symmetric random walks on groups},
   JOURNAL = {Trans. Amer. Math. Soc.},
  FJOURNAL = {Transactions of the American Mathematical Society},
    VOLUME = {92},
      YEAR = {1959},
     PAGES = {336--354},
   MRCLASS = {60.00 (20.00)},
  MRNUMBER = {109367},
MRREVIEWER = {H.\ Bergstr\"om},
}

@book {MR548467,
    AUTHOR = {Sprind\v zuk, Vladimir G.},
     TITLE = {Metric theory of {D}iophantine approximations},
    SERIES = {Scripta Series in Mathematics},
    EDITOR = {Silverman, Richard A.},
 PUBLISHER = {V. H. Winston \& Sons, Washington, DC; John Wiley \& Sons, New
              York-Toronto-London},
      YEAR = {1979},
     PAGES = {xiii+156},
      ISBN = {0-470-26706-2},
   MRCLASS = {10K15 (10-02 10Fxx)},
  MRNUMBER = {548467},
}

@article {MR1652916,
    AUTHOR = {Kleinbock, D. Y. and Margulis, G. A.},
     TITLE = {Flows on homogeneous spaces and {D}iophantine approximation on
              manifolds},
   JOURNAL = {Ann. of Math. (2)},
  FJOURNAL = {Annals of Mathematics. Second Series},
    VOLUME = {148},
      YEAR = {1998},
    NUMBER = {1},
     PAGES = {339--360},
   MRCLASS = {11J83 (22E40)},
  MRNUMBER = {1652916},
MRREVIEWER = {Alexander\ Starkov},
}

@article{shmerkin2016absolute,
  title={Absolute continuity of self-similar measures, their projections and convolutions},
  author={Shmerkin, Pablo and Solomyak, Boris},
  journal={Transactions of the American Mathematical Society},
  volume={368},
  number={7},
  pages={5125--5151},
  year={2016}
}

@inproceedings{shmerkin2016absolute2,
  title={Absolute continuity of complex Bernoulli convolutions},
  author={Shmerkin, Pablo and Solomyak, Boris},
  booktitle={Mathematical Proceedings of the Cambridge Philosophical Society},
  volume={161},
  number={3},
  pages={435--453},
  year={2016},
  organization={Cambridge University Press}
}

@inproceedings{varju2018recent,
  title={Recent progress on Bernoulli convolutions},
  author={Varj{\'u}, P{\'e}ter P},
  booktitle={European Congress of Mathematics},
  pages={847--867},
  year={2018},
  organization={European Mathematical Society-EMS-Publishing House GmbH}
}

@article {BVirred,
    AUTHOR = {Breuillard, Emmanuel and Varj\'u, P\'eter P.},
     TITLE = {Irreducibility of random polynomials of large degree},
   JOURNAL = {Acta Math.},
  FJOURNAL = {Acta Mathematica},
    VOLUME = {223},
      YEAR = {2019},
    NUMBER = {2},
     PAGES = {195--249},
}

@article {BVconv,
    AUTHOR = {Breuillard, Emmanuel and Varj\'u, P\'eter P.},
     TITLE = {On the dimension of {B}ernoulli convolutions},
   JOURNAL = {Ann. Probab.},
  FJOURNAL = {The Annals of Probability},
    VOLUME = {47},
      YEAR = {2019},
    NUMBER = {4},
     PAGES = {2582--2617},
}

@article {polyirr,
    AUTHOR = {Bary-Soroker, Lior and Koukoulopoulos, Dimitris and Kozma,
              Gady},
     TITLE = {Irreducibility of random polynomials: general measures},
   JOURNAL = {Invent. Math.},
  FJOURNAL = {Inventiones Mathematicae},
    VOLUME = {233},
      YEAR = {2023},
    NUMBER = {3},
     PAGES = {1041--1120},
}

@article {OdlyzkoPoonen,
    AUTHOR = {Odlyzko, A. M. and Poonen, B.},
     TITLE = {Zeros of polynomials with {$0,1$} coefficients},
   JOURNAL = {Enseign. Math. (2)},
  FJOURNAL = {L'Enseignement Math\'ematique. Revue Internationale. 2e
              S\'erie},
    VOLUME = {39},
      YEAR = {1993},
    NUMBER = {3-4},
     PAGES = {317--348},
}

@article {Kozmairred,
    AUTHOR = {Bary-Soroker, Lior and Kozma, Gady},
     TITLE = {Irreducible polynomials of bounded height},
   JOURNAL = {Duke Math. J.},
  FJOURNAL = {Duke Mathematical Journal},
    VOLUME = {169},
      YEAR = {2020},
    NUMBER = {4},
     PAGES = {579--598},
}

@article{BVcut,
  title={Cut-off phenomenon for the $ax+b$ Markov chain over a finite field},
  author={Breuillard, Emmanuel and Varj{\'u}, P{\'e}ter P},
  journal={Probability Theory and Related Fields},
  volume={184},
  number={1},
  pages={85--113},
  year={2022},
  publisher={Springer}
}

@article{BVent,
  title={Entropy of Bernoulli convolutions and uniform exponential growth for linear groups},
  author={Breuillard, Emmanuel and Varj{\'u}, P{\'e}ter P},
  journal={Journal d'Analyse Math{\'e}matique},
  volume={140},
  number={2},
  pages={443--481},
  year={2020},
  publisher={Springer}
}

@article {MR4637451,
    AUTHOR = {He, Jimmy and Pham, Huy Tuan and Xu, Max Wenqiang},
     TITLE = {Universality for low-degree factors of random polynomials over
              finite fields},
   JOURNAL = {Int. Math. Res. Not. IMRN},
  FJOURNAL = {International Mathematics Research Notices. IMRN},
      YEAR = {2023},
    NUMBER = {17},
     PAGES = {14752--14794},
}

@article {TaoVu,
    AUTHOR = {Tao, Terence and Vu, Van},
     TITLE = {Local universality of zeroes of random polynomials},
   JOURNAL = {Int. Math. Res. Not. IMRN},
  FJOURNAL = {International Mathematics Research Notices. IMRN},
      YEAR = {2015},
    NUMBER = {13},
     PAGES = {5053--5139},
}

@article{nguyen2011optimal,
  title={Optimal inverse Littlewood--Offord theorems},
  author={Nguyen, Hoi and Vu, Van},
  journal={Advances in Mathematics},
  volume={226},
  number={6},
  pages={5298--5319},
  year={2011},
  publisher={Elsevier}
}

@incollection {MR498417,
    AUTHOR = {Baker, A.},
     TITLE = {The theory of linear forms in logarithms},
 BOOKTITLE = {Transcendence theory: advances and applications ({P}roc.
              {C}onf., {U}niv. {C}ambridge, {C}ambridge, 1976)},
     PAGES = {1--27},
 PUBLISHER = {Academic Press, London-New York},
      YEAR = {1977},
      ISBN = {0-12-074350-7},
   MRCLASS = {10F35},
  MRNUMBER = {498417},
MRREVIEWER = {Hans\ Peter\ Schlickewei},
}

@article {MR220680,
    AUTHOR = {Baker, A.},
     TITLE = {Linear forms in the logarithms of algebraic numbers. {I},
              {II}, {III}},
   JOURNAL = {Mathematika},
  FJOURNAL = {Mathematika. A Journal of Pure and Applied Mathematics},
    VOLUME = {13},
      YEAR = {1966},
     PAGES = {204--216; ibid. 14 (1967), 102--107; ibid. 14 (1967),
              220--228},
}

@book {MR2382891,
    AUTHOR = {Baker, A. and W\"ustholz, G.},
     TITLE = {Logarithmic forms and {D}iophantine geometry},
    SERIES = {New Mathematical Monographs},
    VOLUME = {9},
 PUBLISHER = {Cambridge University Press, Cambridge},
      YEAR = {2007},
     PAGES = {x+198},
      ISBN = {978-0-521-88268-2},
   MRCLASS = {11-02 (11G50 11J81 11J86)},
  MRNUMBER = {2382891},
MRREVIEWER = {Yuri\ Bilu},
}

\end{document}